\documentclass[12pt,a4paper,leqno]{article}
\usepackage[german]{babel}
\usepackage{amscd,amsmath,amssymb,amsfonts}

\renewcommand{\thefootnote}{\fnsymbol{footnote}}
\newtheorem{theorem}[subsubsection]{Theorem}
\newtheorem{proposition}[subsubsection]{Proposition}
\newtheorem{lemma}[subsubsection]{Lemma}
\newtheorem{definition}[subsubsection]{Definition}

\newenvironment{proof}{\medskip\noindent\emph{Proof.}}{\medbreak}
\newenvironment{remark}{\bigskip
\noindent
\addtocounter{subsubsection}{1}
\textbf{Remark~\arabic{section}.\arabic{subsection}.\arabic{subsubsection}}}{\bigbreak}
\newenvironment{remarks}{\bigskip
\noindent
\addtocounter{subsubsection}{1}
\textbf{Remarks~\arabic{section}.\arabic{subsection}.\arabic{subsubsection}}}{\bigbreak}
\newtheorem{corollary}[subsubsection]{Corollary}

\def\z {\mathbb Z}
\def\n {\mathbb N}

\def\ov {\overline}
\def\ni {\noindent}
\def\u {\underline}

\def\cA{\mathcal A}
\def\cG{\mathcal G}
\def\hG{\hat{\mathcal G}}
\def\cT{\mathcal T}
\def\cM{\mathcal M}
\def\cN{\mathcal N}
\def\cO{\mathcal O}
\def\cC{\mathcal C}

\def\ker{\mathrm{ker}}
\def\im{\mathrm{im}}

\def\toto{\rightleftarrows}
\def\da{\downarrow}

\def\qaq{{\quad\mbox{ and }\quad}}

\input xy
\xyoption{all}

\begin{document}

\centerline{\textbf{\large Frobenius gauges and a new theory of $p$-torsion sheaves in
characteristic $p$. I}}

\vskip 5mm

\centerline{by Jean-Marc Fontaine\footnote[1]{Universit\'{e}
Paris-Sud} and Uwe Jannsen\footnote[7]{Universit\"{a}t Regensburg}}

%\centerline{(June 08, 2012)}

\vskip 10mm

\section{Introduction.}\label{sec-introduction}

Let $k$ be a field, $k_s$ a separable closure of $k$ and $G_k={\rm
Gal}(k_s/k)$. Let $X$ be a proper and smooth variety over $k$. This
defines a
morphism of topoi $\pi:X_{\acute{e}t}\to (\mbox{Spec } k)_{\mbox{\scriptsize \'{e}t}} $. Let $l$ be a prime
number and  $F$ be any finite $l$-torsion abelian sheaf over $(\mbox{Spec } k)_{\mbox{\scriptsize \'{e}t}}$, such as $\mathbb Z/l^n\mathbb Z$ or $\mu_{l^n}$ (with $n\in\mathbb N$), then, for any
$j\in\mathbb N$, ${\mathcal F}=R^j\pi_{*}\pi^{*} F$ is again a finite abelian
sheaf over $(\mbox{Spec } k)_{\mbox{\scriptsize \'{e}t}}$.
To know ${\mathcal F}$ amounts to the same as knowing its {\it Galois
module}, that is the
finite abelian group ${\mathcal F}(k_s)$ (the fiber of ${\mathcal F}$ at the
geometric point
$k_s$) with its natural linear and discret action of $G_k$.

If $l$ is different from the characteristic of $k$, this gives rises
to a nice theory,
the {\it \'{e}tale $l$-adic cohomology} which has many useful applications.

Assume now $k$ is a perfect field of characteristic $p>0$ and $l=p$.
Then the \'{e}tale
topology is not big enough for many applications. For instance, the
restriction of $\mu_{p^n}$ to the small \'{e}tale site of $k$ is just
the trivial sheaf.
If we now view $\pi$ as a morphism for the fppf topology
$$\pi : X_{flat}\to (\mbox{Spec } k)_{flat}$$
then ${\mathcal F} = R^1\pi_{*}\pi^{*}\mu_{p^n}$ is a finite commutative
group scheme over
$k$.
To know ${\mathcal F}$ amounts to the same as knowing its {\it
Dieudonn\'{e} module}, an
elementary object whose definition involves some linear algebra.

Now $R^1\pi_{*}\pi^{*}\mu_{p^n}$ is a {\it good} object to consider.
We could also have
consider $R^1\pi_{*}\pi^{*}\mathbb Z/p^n\mathbb Z$ but it would be too small for
some applications. Too
small as well would be the sheaves
$R^j\pi_{*}\pi^{*}\mu_{p^n}$ for $j>1$. Roughly speaking we want to
introduce some
sheaves of $p$-torsion, the $S_n^r$ (for $n,r\in\mathbb N$ (we shall have
$S_n^0=\mathbb Z/p^n\mathbb Z$ and
$S_n^1=\mu_{p^n}$) and to consider the $R^j\pi_{*}\pi^{*}S_n^r$.
They will belong to a
nice class of $p$-torsion sheaves which are classified by their {\it
Dieudonn\'{e}
modules} (elementary objects generalizing the classical Dieudonn\'{e}
modules).
Technically,

i)  we will have to use a topology which is weaker than the flat
topology but also stronger than the \'{e}tale topology,

ii) the generalized Diedonn\'e modules give more complicated objects, so-called
gauges, or rather $\varphi$-(or Frobenius-)gauges.

\medskip
The aim of this paper is to give the definition of these $\varphi$-gauges,
related them to previous constructions like Dieudonn\'{e} modules,
$F$-zips, or displays, and to define a cohomology theory with values in $\varphi$-gauges,
which refines crystalline cohomology. For examples and applications we will
rather concentrate on the case of varieties over fields.

\medskip
In following papers we will construct and discuss the functors between $\varphi$-gauges and
certain $p$-torsion sheaves for the syntomic cohomology, will develop a relative theory,
and will discuss relations with $p$-adic Hodge theory over discrete valuation rings.

\bigskip
{\it Remark}: Let $K$ be a field of characteristic $0$, complete with
respect to a
discrete valuation, with perfect residue field $k$ of characteristic
$p$. Let $\overline K$ be
an algebraic closure of $K$ and $G_K={\rm Gal}(\overline K/K)$. Let
$V$ be a crystalline representation of $G_K$ with non negative
Hodge-Tate weights and $T$
a $G_K$-stable $\mathbb Z_p$-lattice of $V$. Then $V/T$ may be viewed in a
natural way as one of
the fibers of a sheaf $\Gamma$ for the syntomic-\'{e}tale topology
over $\mbox{Spec } {\mathcal
O}_K$ (for instance, if the Hodge-Tate weights are $0$ and $1$ this
is a Barsotti-Tate
group (or $p$-divisible group) over ${\mathcal O}_K$. This $\Gamma$ as a
special fiber
$\Gamma_k$ and, roughly speaking the kernel of the multiplication by
$p^n$ on $\Gamma_k$
is one of these nice $p$-torsion sheaves that we are constructing.
Somehow, the theory we
develop here is the {\it special fiber of classical $p$-adic Hodge theory}.

One application we expect from this theory is to deformation of
Galois representations
coming from algebraic geometry. As we just said, these may often be
extended in a natural
way to a $p$-adic sheaf for the syntomic-\'{e}tale topology. The
knowledge of the $p$-adic
sheaf is equivalent to the knowledge of the Galois representation.
But when we take
sub-quotients killed by a power of $p$, this is no more true in
general and it may be
wise to look at the deformations of the sheaves rather than just at
the deformations of
the Galois representations. In the case of Barsotti-Tate groups, this
idea has already been used by Kisin \cite{Ki}.
Long time ago, just after our joint work on $p$-adic Hodge theory,
William Messing and one of us (JMF) started to think about this kind of
things. What follows is just a continuation of this old work which has
never been completed and we want to thank Bill heartily for the old
and new discussions we had with him on that.

\bigskip

This paper is organized as follows: In section 1, we introduce the notion
of gauges, $\varphi$-modules, and $\varphi$-gauges, which are the basic objects from
linear algebra which give rise to the notion of generalized
Dieudonn\'{e} modules. In section 2, we study how these structures arise from (virtual) $W$-crystals,
where $W = W(k)$ is the ring of Witt vectors for a perfect field of characteristic $p$, and we
discuss properties of the new category in this case. In sections 3 and 4 we consider arbitrary
schemes in positive characteristic and show that our theory contains and extends the theories of
$F$-zips and displays. In sections we discuss the different Grothendieck
topologies that we are going to use and their basic properties. In
section 6 we recall the definitions and properties of the syntomic sheaves of
rings $\mathcal O^{cris}_n$ and we explain how one can use these rings to get a gauge of rings ${\mathcal G}$,
 also called the universal gauge, which is central for our theory.
In section 7 we define the gauge cohomology (cohomology theory with values in the category of gauges)
which is a refinement of crystaline cohomology,
and we give some first properties.

\section{Graded objects, gauges, $\varphi$-modules, $\varphi$-gauges, and $\varphi$-rings.}\label{sec-1}

\subsection{\it Graded objects and $p$-gauges}\label{ssec-1.1}
By a graded object in an abelian category $\cA$ we mean a $\z$-graded object, which is
just a collection $A=(A^n)$ of objects indexed by $\z$. If direct sums exist in $\cA$, we
may also think of the direct sum $\oplus_nA^n$. A morphism $f: A \rightarrow B$ of graded
objects of degree $d$ is a collection of morphisms $f_n: A^n \rightarrow B^{n+d}$.
A morphism of graded objects is a morphism of degree 0. Graded objects in $\cA$ form again an abelian category.

\medskip
Fix a prime $p$. A $p$-gauge in $\cA$ is a graded object $M$ in $\cA$ together with
a morphism $f$ of degree 1 and a morphism $v$ of degree $-1$
$$
\ldots \toto M^{r-1} \mathop{\toto}\limits^{f}_{v} M^r
\mathop{\toto}\limits^{f}_{v} M^{r+1} \toto \ldots
$$
with $fv = p = vf$. Morphisms of $p$-gauges are morphisms $\alpha$ of graded object which
are compatible with $f$ and $v$ (i.e., $\alpha f = f \alpha$ and $\alpha v = v \alpha$).

\medskip
If, for example, $\cA$ is the category of modules over a ring, then a $p$-gauge is
simply a module over the commutative graded ring $D(R)=R[f,v]/(fv-p)$, where $R[f,v]$ is the
free graded ring generated by $f$ in degree 1 and $v$ in degree $-1$.

\medskip
For $-\infty \leq a \leq b \leq \infty$ we say that the $p$-gauge $(M,f,v)$ is concentrated in
the interval $[a,b]$, if $v$ is an isomorphism to the left of $M^a$ and $f$ is an isomorphism
to the right of $M^b$. If $a$ and $b$ are finite, then such a $p$-gauge is just determined by
the finite diagram
$$
M^a \toto \ldots \toto M^{r-1} \mathop{\toto}\limits^{f}_{v} M^r
\mathop{\toto}\limits^{f}_{v} M^{r+1} \toto \ldots \toto M^b\,,
$$
because everything is determined outside the interval $[a,b]$.

\medskip Call a $p$-gauge over a ring $R$ (i.e., in the category of $R$-modules) of finite type,
if it is finitely generated as a module over $D(R)=R[f,v]/(fv-p)$.

\begin{lemma} Let $R$ be a noetherian ring, and let $M$ be a $p$-gauge of $R$-modules.
Then the following conditions are equivalent.

\smallskip (a) $M$ is of finite type.

\smallskip (b) Each $M^r$ is finitely generated as an $R$-module, and $M$ is concentrated in
a finite interval (i.e., $f: M^r \rightarrow M^{r+1}$ is an isomorphism for $r>>0$ and
$v: M^r \rightarrow M^{r-1}$ is an isomorphism for $r<< 0$).
\end{lemma}

\begin{proof} (Compare the positively graded case in \cite{GW}, Lemma 13.10.)
Assume (a) and let $m_1,\ldots,m_r$ be generators of $M$, without restriction each $m_i$ homogenous
of degree $d_i\in\z$, say. Let $d_{min}$ be the minimum of the $d_i$, and let
$d_{max}$ be their maximum. Then every element of $M^n$ is a $R$-linear combination of
the elements in the set
$$
S_n = \{\, f^am_i\,  \mid \, a \geq 0,\, n = a + d_i \,\}\cup \{ \, v^bm_i\, \mid \, b \geq 0,\, n = d_i - b \,\}\,.
$$
(Note that $fv = p =vf$.)
Since
$$
d_{min} - n \leq b = d_i-n \leq d_{max} - n \quad\mbox{ and }\quad n - d_{max} \leq a - d_i \leq n - d_{min}\,,
$$
these are finitely many elements, which shows the first claim in (a).

For the second claim we first note that $f: M^n \rightarrow M^{n+1}$ is surjective for $n\geq d_{max}$.
In fact, for the elements $v^bm_i$ with $n+1 = d_i-b$ and $b\geq 0$ in the generating set $S_{n+1}$ above
we would have $d_i-b = n + 1 \geq d_{max} +1$, i.e., $d_i \geq d_{max} + 1$, a contradiction.
Hence these elements do not appear. Moreover, for the elements $f^am_i$ with $a + d_i = n+1 \geq d_{max} + 1$
we must have $a \geq d_{max} - d_i + 1 \geq 1$.

Next, for $d=d_{max}$ the sequence of surjections
$M^d \mathop{\twoheadrightarrow}\limits^{f} M^{d+1} \mathop{\twoheadrightarrow}\limits^{f} M^{d+2} \mathop{\twoheadrightarrow}\limits^{f}\ldots $
becomes stationary, because $R$ is noetherian and all $M^i$ are finitely generated. Thus $f: M^n \rightarrow M^{n+1}$
is surjective for $n>>0$. In a similar (dual) way one proves that $v: M^{n+1} \rightarrow M^n$ is an isomorphism for $n<<0$.

The converse implication from (b) to (a) is easier: If $M$ is concentrated in the finite interval $[a,b]$,
and $F$ is a finite generating set for the $R$-module $M^a\oplus M^{a+1} \oplus \ldots \oplus M^{b-1} \oplus M^b$,
then $F$ is a generating set for the graded $D(R)$-module $M$.
\end{proof}

\medskip
In the following we will always fix a prime $p$ and will omit it in the notation.
Of course we could make a more general definition and replace the multiplication by $p$
by any natural transformation $t: id_\cA \rightarrow id_\cA$.

\subsection{\it $\varphi$-modules and $\varphi$-gauges}\label{ssec-1.2}

\medskip
For any gauge $(M,f,v)$ in $\cA$ we define
$$
M^{+\infty}=\lim\limits_{\stackrel{\textstyle \longrightarrow}{r\mapsto +\infty}}M^{r}\quad\hbox{ and }\quad
M^{-\infty}=\lim\limits_{\stackrel{\textstyle \longrightarrow}{r\mapsto -\infty}}M^r\,,
$$
where the transition morphisms are given by the morphisms $f$ and the morphisms $v$, respectively,
and where these objects either exist as direct limits in $\cA$ or as objects in the $Ind$-category of $\cA$.

\medskip
Now let $\sigma: \cA \rightarrow \cA$ be some endomorphism of $\cA$. A $\varphi$-module (with respect to $\sigma$)
is a gauge $(M,f,v)$ together with a morphisms
$$
\varphi:\;\sigma(M^{+\infty}) \longrightarrow M^{-\infty}\,.
$$
Morphisms of $\varphi$-modules are morphisms $\alpha$ of gauges which are compatible with $\varphi$
(i.e., for the morphisms $\alpha^{+\infty}$ and $\alpha^{-\infty}$ induced by $\alpha$ on the limit terms
one has $\varphi \sigma(\alpha^{+\infty}) = \alpha^{-\infty} \varphi$). A $\varphi$-gauge is a $\varphi$-module
for which $\varphi$ is an isomorphism. If the formation of inductive limits is exact, then it is easy to see
that the $\varphi$-modules form an abelian category if $\sigma$ is right-exact, and that the category of
$\varphi$-gauges is abelian if $\sigma$ is exact.

\subsection{\it Tensor products of graded modules and gauges}\label{ssec-1.3}
To fix ideas let ${\mathcal T}$ be a topos.  Assume that ${\mathcal T}$ is the topos
of sheaves over ${\mathcal C}$ for the topology $E$.

If ${\mathcal F}$ is a sheaf of sets (or groups, or ...), a {\it global section
of ${\mathcal F}$} s a collection $(s_U)_{U\in Ob{\mathcal C}}$ such that, for
any morphism $f: V\to U$ of ${\mathcal C}$, we have $f^{-1}(s_U)=s_V$.  The
global sections of ${\mathcal F}$ form a set (resp. a group,...) $\Gamma({\mathcal
F})$. If ${\mathcal C}$ has a final object $S$, we have $\Gamma({\mathcal F})={\mathcal F}(S)$.

Below we will consider groups, rings, modules in $\mathcal T$, but for simplicity, we will
omit $\mathcal T$ and do as if we were just using plain groups, rings, modules (which would
correspond to the trivial topos), by thinking of local sections,.... at least when the extension of the things
considered to the general setting is completely straightforward.
In the whole paper, a graded ring is a commutative graded ring
with grading indexed by $\mathbb Z$.

Let $R=\oplus_{r\in\mathbb Z}R^{r}$ such a graded ring .
A {\it graded $R$-module} is an $R$-module $M$ together with a decompostion
$M=\oplus_{r\in\mathbb Z}M^r$ of
$M$ in to a direct sum of abelian groups such that, if $r,s\in\mathbb Z$,
$\lambda\in R^r$ and
$x\in M^s$, then $\lambda x\in M^{r+s}$. The graded $R$-modules form
a $\Gamma(R^0)$-linear abelian
category.

Let $M$ and $N$ be two graded $R$-modules. For $r\in\mathbb Z$, set
$(M\otimes_{R^0}N)^r=\oplus_{i+j=r}M^i\otimes_{R^0}N^j$. Let $L^r$ be
the sub-group of
$(M\otimes_{R^0}N)^r$ (locally) generated by the $\lambda x\otimes
y-x\otimes\lambda y$, for
$\lambda\in R^i$, $x\in M^j$, $y\in N^k$ and $i+j+k=r$. We have
$M\otimes_{R^0}N=\oplus_{r\in\mathbb Z}(M\otimes_{R^0}N)^r$ and
$$M\otimes_{R}N=\oplus_{r\in\mathbb Z}(M\otimes_{R}N)^r\hbox{ with }(M\otimes_{R}N)^r
=(M\otimes_{R^0}N)^r/L^r$$
This endows $M\otimes_RN$ with a structure of graded $R$-module. In this way,
graded-$R$-modules form a tensor category.

If $M$ is a graded $R$-module, for any $i\in\mathbb Z$,{\it  the $i$-th
Tate twist of $M$} is the
graded $R$-module whose underlying $R$-Module is $M$ and with
$M(i)^r=M^{r+i}$. The
functor $M\mapsto M(1)$, from the category of graded $R$-Modules to itself,
is an equivalence of categories, with $M\mapsto M(-1)$ as a quasi-inverse.
We have $M(0)=M$, $M(i+j)=M(i)(j)$ ($\forall i,j\in\mathbb Z$) and
$M(i)=R(i)\otimes_RM$.

A  {\it free graded $R$-module of rank $1$} is a graded $R$-module
isomorphic to an $R(i)$, for some $i\in\mathbb Z$. A graded $R$-module $M$
is called {\it free} if it can be written
as a direct sum of free $R$-modules of rank $1$.
For any graded $R$-module $M$ there is a canonical bijection
$$
Hom(R(i),M) \mathop{\longrightarrow}\limits^{\sim} \Gamma(M^{-i})\,,
$$
sending a morphism $R(i) \rightarrow M$ to the image of $1 \in \Gamma(R(i)^{-i})=R$.
Thus an
$R$-Module can be written as a quotient of a free graded $R$-module
if and only if it is generated by global sections (this is always the
case, if the topos is trivial !).

If $R$ is a graded ring, and $I \subset R$ is a graded ideal, i.e., generated by
homogeneous elements, then $R/I$ is naturally a graded ring. We can apply this to
the category of gauges: if $R_0$ is a ring (in $\cT$), then the category of $R_0$-gauges
in $\cT$ is equivalent to the category of graded $D(R_0)$-modules, where $D(R_0)=R_0[f,v]/(fv-p)$
is the graded ring (in $\cT$) defined similarly as in section 1.1. Therefore there is a natural
tensor product $M\otimes N$ of $R_0$-gauges $M$ and $N$ in $\cT$, defined as the tensor product $M\otimes_{D(R_0)}N$.

One easily sees that, for $R_0$-gauges $M$ and $N$, one has canonical isomorphisms
$$
(M\otimes_{D(R_0)} N)^{+\infty}\cong M^{+\infty}\otimes_{R_0}N^{+\infty} \quad\mbox{  and  }\quad
(M\otimes_{D(R_0)} N)^{-\infty}\cong M^{-\infty}\otimes_{R_0} N^{-\infty}\,.
$$
As a consequence, if $\sigma$ is an endomorphism of the category of $R_0$-modules, and is
a tensor morphism, there is a canonical tensor product on the category of $\varphi$-gauges
of $R_0$-modules with respect to $\sigma$, by endowing $M\times N$ with the following $\varphi$:
$$
\sigma(M^{+\infty}\otimes_{R_0}M^{+\infty}) \cong \sigma(M^{+\infty})\otimes_{R_0}\sigma(M^{+\infty}) \mathop{\longrightarrow}\limits^{\varphi_M\otimes\varphi_N} M^{-\infty}\otimes_{R_0}N^{-\infty}\,.
$$
Obviously, this respects the subcategory of $\varphi$-gauges

\subsection{\it $\varphi$-rings, and $\varphi$-modules and $\varphi$-gauges over them.}\label{ssec-1.4}

\medskip
We need a certain generalization of the considerations in the previous section.
As there, we consider objects (rings, modules, etc.) in some topos $\cT$, and suppress the mentioning of $\cT$.
Consider a triple $(R,f,v)$ where $R=\oplus_{n\in\mathbb Z}R^n$ is a $\z$-graded
commutative ring (in $\cT$) with $f\in \Gamma(R^1)$ and $v\in \Gamma(R^{-1})$.
Set
$$
R^{+\infty}=R/(f-1)\quad\hbox{ and }\quad R^{-\infty}=R/(v-1)
$$
Observe that we may identify these two rings, as $R^0$-modules, to
the direct limits
$$R^{+\infty}=\lim\limits_{\stackrel{\textstyle \longrightarrow}{r\mapsto +\infty}}R^{r}\quad\hbox{ and }\quad
R^{-\infty}=\lim\limits_{\stackrel{\textstyle \longrightarrow}{r\mapsto -\infty}}R^r$$
(the transition maps being given by multiplication by $f$ (resp $v$).

We define {\it  a $\varphi$-ring}  as a quadruple $(R,f,v,\varphi)$
with $(R,f,v)$ as above and
$$
\varphi:R^{+\infty}\longrightarrow R^{-\infty}
$$
a morphism of rings.
If this is an isomorphism, we call $(R,f,v,\varphi)$ a perfect $\varphi$-ring.

\bigskip
Let $R=(R,f,v,\varphi)$ a $\varphi$-ring. If $M$ is a graded
$R$-module, we may consider the $R^{+\infty}$-module
$M^{+\infty}=R^{+\infty}\otimes_R M= M/(f-1)M$ and the $R^{-\infty}$-Module
$M^{-\infty}=R^{-\infty}\otimes_R M= M/(v-1)M$. Observe that, as
$R^0$-modules, we also have the identifications
$$
M^{+\infty}=\lim\limits_{\stackrel{\textstyle \longrightarrow}{r\mapsto +\infty}}M^{r}\quad\hbox{ and }\quad
M^{-\infty}=\lim\limits_{\stackrel{\textstyle \longrightarrow}{r\mapsto -\infty}}M^r
$$

A {\it  $\varphi$-$R$-module} is a pair $(M,\varphi)$ where $M$ is a graded
$R$-module and
$$
\varphi:M^{+\infty} \longrightarrow M^{-\infty}
$$
is a morphism of groups such that
$\varphi(\lambda x)=\varphi(\lambda)\varphi(x)$ for $\lambda\in
R^{+\infty}$ and $x\in
M^{+\infty}$.

A {\it $\varphi$-$R$-gauge} is a  $\varphi$-$R$-module $(M,\varphi_M)$ such that
the canonical morphism of $R^{-\infty}$-modules
$$
\varphi'_M: R^{-\infty} {}_{\varphi_R}\mathop{\otimes}\limits_{R^{+\infty}}M^{+\infty}\longrightarrow M^{-\infty}
$$
induced by $\varphi_M$ is an isomorphism. If $(R,\varphi_R)$ is a perfect $\varphi$-ring,
this holds if and only if $\varphi_M$ is an isomorphism.

With obvious definitions of morphisms, the graded $R$-modules and the
$\varphi$-$R$-modules are abelian categories with enough injectives.
As a full sub-category of the category of  $\varphi$-$R$-modules, the
category of $\varphi$-$R$-gauges is stable under direct sums and direct factors.
If $R$ is a perfect $\varphi$-ring, it is also stable under kernels and cokernels
and therefore also abelian (here we use the exactness of the formation
of $M^{+\infty}$ and $M^{-\infty}$).

If $M$ and $N$ are two  $\varphi$-$R$-modules, we have
$$
(M\otimes_RN)^{+\infty}=M^{+\infty}\otimes_{R^{+\infty}}N^{+\infty}
\quad\mbox{ and }\quad
(M\otimes_RN)^{-\infty}=M^{-\infty}\otimes_{R^{-\infty}}N^{-\infty}\,.
$$
Therefore, the morphism $\varphi\otimes\varphi$ endows $M\otimes_RN$
with a structure of  $\varphi$-$R$-Module (which is a $\varphi$-gauge if $M$ and
$N$ are $\varphi$-gauges).  With this tensor product,  $\varphi$-$R$-modules and
$\varphi$-$R$-gauges become tensor categories.

For any graded $R$-module $M$ and any $i\in\mathbb Z$, we have
$M(i)^{+\infty}=M^{+\infty}$ and $M(i)^{-\infty}=M^{-\infty}$. This
allows us to extend the definition of Tate twists to
$\varphi$-$R$-modules and $\varphi$-$R$-gauges in an obvious way.

\section{\bf Gauges over a perfect field.}\label{sec-2}

\subsection{\it Preliminaries}\label{ssec-2.1}
\medskip
Assume $R$ is a graded ring with $f\in R^1$ and $v\in R^{-1}$ such
that $R=R^{0}[f,v]$.

If $fv$ is not a zero divisor in $R^{0}$, the natural maps
$R^0\to R^{+\infty}$ and $R^0\to R^{-\infty}$ are
isomorphisms and to give a map $\varphi:R^{+\infty}\to R^{-\infty}$
such that $(R,f,v,\varphi)$ is a $\varphi$-ring is the same as giving
an automorphism $\varphi$ of the ring $R$.

If moreover $fv$ is invertible in $R^0$, the correspondence $M\mapsto
M^0$ induces an equivalence between the category of graded $R$-modules
and the category of $R^0$-modules. Let $R^0[\varphi]$ the
(non-commutative if $\varphi\not=id_{R^{0}}$) ring generated by $R^0$ and an
element $\varphi$, with the relation that
$\varphi\lambda=\varphi(\lambda).\varphi$, for $\lambda\in R^0$.
Similarly, let $R^0[\varphi,\varphi^{-1}]$ the ring generated over
$R^0[\varphi]$ by an element $\varphi^{-1}$ with the relations
$\varphi.\varphi^{-1}=\varphi^{-1}\varphi=1$ and
$\lambda\varphi^{-1}=\varphi^{-1}\varphi(\lambda)$ for $\lambda\in R^0$.
The previous equivalence of categories induces an equivalence between

-- $\varphi$-$R$-modules and left $R^0[\varphi]$-modules,

-- $\varphi$-$R$-gauges and left $R^0[\varphi,\varphi^{-1}]$-modules.

In these equivalences, the tensor product becomes the tensor product
over $R^0$.
\bigskip

The situation is slightly more complicated when $fv$ is not
invertible. This is the situation for our generalized Dieudonn\'{e}
modules~: We chose a perfect field $k$ of characteristic $p$ and we let
$W=W(k)$ be the ring of Witt vectors over $k$ and, for any $n \in \mathbb N$,
we let $W_n = W_n(k) = W/p^n$ be the ring of $n$-th truncated Witt vectors.
To cover both cases, we write $W_n$ for $n \in \mathbb N \cup \{\infty\}$, where $W_\infty:= W$.
Then $W_n$-gauges, i.e., gauges of $W_n$-modules, are simply graded modules
over the ring $D_n =D_n(k) = W_n[f,v]/(fv-p)$.

% our graded ring which is defined by
%the conditions hat $D^0=W$ and $fv=p$. By a {\it graded module over
%$k$}, we mean a graded $D$-module. The graded modules over $k$
%form an abelian $W$-linear category.

%A graded module $M$ over $k$ may be viewed as a diagram of $W$-modules:
%$$\ldots M^{r-1} \mathop{\phantom{v}}\limits^{\stackrel{\scriptstyle f}{\textstyle\rightarrow}}_{\stackrel{\textstyle\leftarrow}{v}}  M^r %M^{r+1}\mathop{\phantom{v}}\limits^{\stackrel{\scriptstyle f}{\textstyle\rightarrow}}_{\stackrel{\textstyle\leftarrow}{v}}\ldots$$
%such that $fv=vf=p$.

\medskip
We turn $D_n$ into a perfect $\varphi$-ring by taking for $\varphi$ the absolute Frobenius
$\sigma: D_n^{+\infty} =W_n \longrightarrow W_n = D^{-\infty}$, which is an isomorphism
by perfectness of $k$.

%In this paper, a {\it $\varphi$-module over $k$} is a
%$(D,\varphi)$-module, hence a graded module over $k$ equipped with a
%map $\varphi:M^{+\infty}\to M^{-\infty}$, semi-linear with respect to
%the absolute Frobenius on $W$.  A {\it gauge over $k$} is a
%$D$-gauge, hence it is a $\varphi$-module $M$ over $k$ such that
%$\varphi:M^{+\infty}\to M^{-\infty}$ is an isomorphism.

Hence a $\varphi$-$W_n$-module is a $W_n$-gauge $(M,f,v)$ together with
a group homomorphism $\varphi: M^{+\infty} \rightarrow M^{-\infty}$ which is
semi-linear with respect to the absolute Frobenius $\sigma$ on $W_n$.

%A {\it tft graded module} (resp. {\it $\varphi$-module}, resp. {\it gauge})
%{\it over $k$} is a graded module (resp. a
%$\varphi$-module, resp. a gauge) over $k$ such that the underlying
%$D$-module is
%of $p$-torsion and of finite type. We denote by ${\mathcal Gr}(k)$ (resp.
%${\mathcal M}(k)$, ${\mathcal G}(k)$) the
%category of graded modules (resp. $\varphi$-modules, resp. gauges) over
%$k$ and by $\mathcal Gr_{tft}(k)$ (resp. $\mathcal M_{tft}(k)$, ${\mathcal G}_{tft}(k)$) the
%full sub-category whose objects are tft.

%If $M$ is a tft graded module over $k$, there
%are integers $a\leq b$ such that

%-- multiplication by $f$ induces an isomorphism $M^{r}\to M^{r+1}$,
%for $r\geq b$,

%-- multiplication by $v$ induces an isomorphism $M^{r}\to M^{r-1}$,
%for $r\leq a$.

%We denote $\mathcal Gr^{[a,b]}_{tft}$ (resp. $\mathcal M^{[a,b]}_{tft}$,resp. $\mathcal G^{[a,b]}_{tft}$) the full subcategory of
%those tft
%graded modules (resp.
%$\varphi$-modules, resp. gauges) over $k$ which
%satisfy the previous condition.

We say that a $W_n$-gauge or $\varphi$-$W_n$-gauge $M$ is of finite type, if the
associated $D_n$-module is finitely generated. As we have seen in section \ref{ssec-1.1},
$M$ is then concentrated in a finite interval $[a,b]$ and is just given by the
finite diagram of finitely generated $W_n$-modules
$$
M^a \toto \ldots \toto M^{r-1} \mathop{\toto}\limits^{f}_{v} M^r
\mathop{\toto}\limits^{f}_{v} M^{r+1} \toto \ldots \toto M^b\,,
$$
such that $fv=vf=p$ (For $r>b$ , we use the multiplication by $f^{b-r}$ to identify
$M^b$ to $M^r$. Similarly, for $r<a$, we use multiplication by
$v^{a-r}$ to identify $M^a$ to $M^{r}$). The structure of a $\varphi$-module is
obtained by adding a $\sigma$-semi-linear map  $\varphi: M^{ b}\to M^a$
(because, we have canonical identifications
$M^{+\infty}=M^b$ and $M^{-\infty}=M^a$). The $\varphi$-module $M$
will be a gauge if and only if $\varphi:M^{b}\to M^a$ is bijective.

\bigskip
Let $\cG^{[a,b]}_{ft}(W_n)$ is the category of finite-type $W_n$-gauges which are
concentrated in $[a,b]$. If $M$ is an object of $\cG^{[a,b]}_{ft}(W_n)$ and $N$ an object of
$\cG_{ft}^{[a',b']}(W_n)$,
then $M\otimes N$ is an object of $\cG_{ft}^{[a+a',b+b']}(W_n)$.
If $M$ is an object of $\cG^{[a,b]}_{ft}(W_n)$ and $i\in\mathbb Z$, then $M(i)$ is an object of
$\cG_{ft}^{[a-i,b-i]}(W_n)$.

\subsection{\it The standard construction: $p$-divisibility of Frobenius}\label{ssec-2.2}

\medskip
The idea of gauges is related to the following construction, going back to ideas
of Mazur and Kato. Let $B$ be the fraction field of $W$.

Let $D$ be an isocrystal over $k$, i.e., a finite dimensional $B$-vector space with
a $\sigma$-semi-linear isomorphism $\phi: D \rightarrow D$, and let $M$ be a
lattice in $D$, i.e., a finitely generated $W$-submodule with $M\otimes_WB\cong D$.
We call such object a virtual crystal over $k$ (and a crystal if $\phi(M) \subset M$).
For $r\in \z$ define
$$
M^r = \{\,m \in M \,\mid\, \phi(m) \in p^rM\,\}\,,
$$
let $f: M^r \rightarrow M^{r+1}$ be the multiplication by $p$, and let
$v: M^{r+1} \rightarrow M^r$ be the inclusion. Then $(M^\cdot,f,v)$ is a $W_n$-gauge.
Moreover, by finite generation of $M$ one has integers $a \leq b$ with
$p^bM \subseteq \phi(M) \subseteq p^aM$. The last inclusion implies that $M \subset M^a$
and hence the inclusions $M^r \subset M \subset M^a \subset M$ are isomorphisms
for $r \leq a$. The first inclusion implies that $M^r = p^{r-b}M^b$ for $r \geq b$.
In fact, if $x \in M^r$, i.e., $\phi(x) = p^ry = p^{r-b}p^by$ with $y\in M$, then
$p^by = \phi(z)$ with $z\in M$. This implies $\phi(x) = \phi(p^{r-b}z)$
and hence $x = p^{r-b}z$, where $z \in M^b$. We conclude that the gauge is concentrated
in the interval $[a,b]$.

Moreover, we get a canonical structure of a $\varphi$-$W$-gauge. In fact, we have natural
$\sigma$-semi-linear homomorphisms
$$
\varphi_r: M^r \rightarrow M = M^{-\infty}\,,
$$
by sending $x\in M^r$ to $p^{-r}\phi(x)$.
These are compatible ($\varphi_{r+1}(fx) = p^{-r-1}\varphi(px) = p^{-r}\varphi(x) = \varphi_r(x)$)
and thus define a $\sigma$-semi-linear morphism
$$
\varphi: M^{+\infty} = \lim\limits_{\stackrel{\textstyle \longrightarrow}{r\mapsto +\infty}}M^{r} \longrightarrow M^{-\infty}\,,
$$
which is easily seen to be an isomorphism. Moreover, one sees

\begin{theorem}\label{thm.virtual.crystals.1}
The above construction gives a fully faithful embedding of categories
$$
( \mbox{ virtual crystals } (D,\phi,M) \mbox{ over } k ) \longrightarrow
( \mbox{ finite type } \varphi\mbox{-}W\mbox{-gauges with free components } )\,.
$$
\end{theorem}

Now we want to characterize the essential image of this functor. If $(M,f,v)$ is a gauge,
then we let
$$
f_r: M^r \longrightarrow M^{+\infty} \quad\mbox{ and }\quad v_r: M^r \longrightarrow M^{-\infty}
$$
be the canonical morphisms into the respective inductive limits. We introduce the
following definitions, which will also be of use later.

\begin{definition}\label{def.rigid}
Let $\cA$ be an $\mathbb F_p$-linear abelian category.
A gauge $(M,f,v)$ in $\cA$ is called

\smallskip (a) {\it strict}, if the morphism
$$
(f_r,v_r): M^r \longrightarrow R^{+\infty}\oplus R^{-\infty}
$$
is a monomorphism for all $r\in\z$,

\smallskip (b) {\it quasi-rigid}, if the sequence
$$
M^r \mathop{\longrightarrow}\limits^{f} M^{r+1} \mathop{\longrightarrow}\limits^{v} M^r \mathop{\longrightarrow}\limits^{f} M^{r+1}
$$
is exact for all $r\in\z$,

\smallskip (c) {\it rigid}, if $M$ is strict and quasi-rigid.
\end{definition}

We note that, in general, the notion of (quasi-)rigidity makes sense only if the objects
are annihilated by $p$, since $vf = p = vf$.

\begin{lemma}\label{lem.quasi-rigid}
Let $M$ be a quasi-rigid gauge, and assume that one $M^s$ has finite length.
Then all $M^r$ have finite length and have the same length.
\end{lemma}

\begin{proof}
For the morphisms $f^{(r)}: M^r \rightarrow M^{r+1}$ and $v^{(r+1)}: M^{r+1} \rightarrow M^r$ we have exact sequences
$$
0 \rightarrow \mbox{im}(v^{(r+1)}) \hookrightarrow M^r \mathop{\twoheadrightarrow}\limits^{f} \mbox{im}(f^{(r)})\rightarrow 0
\quad\mbox{ and }\quad
0 \rightarrow \mbox{im}(f^{(r)}) \hookrightarrow M^{r+1} \mathop{\twoheadrightarrow}\limits^{v} \mbox{im}(v^{(r+1)})\rightarrow 0\,.
$$
This implies that $M^r$ and $M^{r+1}$ have the same length, hence the claim.
\end{proof}

Now we consider gauges over a field $k$ of characteristic $p>0$. We have the following
Nakayama-type lemma:

\begin{lemma}\label{lem.nakayama}
Let $M$ be a gauge of finite type over $k$. If $M/(f,v)M = 0$, then $M=0$. As a consequence, if
$m_1,\ldots,m_r$ are homogeneous elements in $M$ whose residue classes generate $M/(f,v)$ as a $k$-vector space,
then these elements generate $M$ (as a $D(k)=k[f,v]/(fv)$-module).
\end{lemma}

\begin{proof}
Assume that $M/(f,v)M=0$ and let $m \in M^s$ for some $s$. Then there exist elements
$m^1_{s-1}\in M^{s-1}$ and $m^2_{s+1}\in M^{s+1}$ with $m = fm^1_{s-1} + vm^2_{s+1}$.
By induction, and noting that $fv = vf = p = 0$, for each $n>0$ we get elements
$m^1_{s-n}\in M^{s-n}$ and $m^2_{s+n}\in M^{s+n}$ with
$$
m = f^nm^1_{s-n} + v^nm^2_{s+n}\,.
$$
We see that $m=0$, since $f=0$ on $M^{s-n}$ for $n>>0$ and $v=0$ on $M^{s+n}$ for $n>>0$,
because $v$ is an isomorphism on $M^r$ for $r<<0$ and $f$ is an isomorphism on $M^r$ for
$r>>0$, and $fv=vf=0$.
The second claim follows in a standard way, by looking at $M/N$, where $N$ is the
sub-$D(k)$-module generated by $m_1,\ldots,m_r$.
\end{proof}

We derive from this the following criterion.

\begin{lemma}\label{lem.free-mod-p}
Let $M$ be a $k$-gauge of finite type. Then the following are equivalent.

\smallskip (a) $M$ is free.

\smallskip (b) $M$ is rigid.

(c) The maps $M^r/v \mathop{\longrightarrow}\limits^{f}M^{r+1}/v$ and
$M^r/f \mathop{\longleftarrow}\limits^{v} M^{r+1}/f$ are injective for all $r$.
\end{lemma}

Here we have used the short notation $M^r/v$ for $(M/(v))^r$ or, explicitly, $M^r/vM^{r+1}$;
similarly for $M^r/f$.

\begin{proof}
Obviously, (b) holds for the free gauge $k = k(0)$ (see the definition in section \ref{ssec-1.3}),
i.e., the module $M=D=k[f,v]/(fv)$, where $M/v$ is a free $k[f]$-module and $M/f$ is a free $k[v]$-module.
In fact, this gauge corresponds to the diagram
$$
\ldots \mathop{\toto}\limits^{0}_{id} k \mathop{\toto}\limits^{0}_{id} k \mathop{\toto}\limits^{id}_{0} k \mathop{\toto}\limits^{id}_{0} \ldots \,,
$$
where the middle $k$ is placed in degree 0. One immediately sees strictness ($\ker(v)\cap\ker(f)=0$)
and quasi-rigidity ($\ker(f)=\im(v)$ and $\ker(v)=\im(f)$) at all places.
Thus (b) holds for gauges $k(i)$ by degree shifting, and for arbitrary free gauges by taking sums.

On the other hand, (b) implies (c). In fact, for the injectivity of $M^r/v \mathop{\rightarrow}\limits^{f} M^{r+1}/v$
assume that $f(x) = v(y) =: a$ for $x\in M^r$ and $y\in M^{r+2}$. By rigidity we have $0 = \ker(v)\cap\ker(f) = \im(f)\cap\im(v)$,
so that $a=0$. Since $\ker(f) = \im(v)$, $f(x)=0$ implies $x \in \im(v)$ as claimed. The injectivity of
$M^{r+1}/f \mathop{\rightarrow}\limits^{v} M^r/f$ follows dually: If $v(y) = f(x)$, then $y\in \ker(v) = \im(f)$.
We also note that (c) immediately implies that $M$ is quasi-rigid: We have a factorization
$M^r/v \mathop{\rightarrow}\limits^{f} M^{r+1} \twoheadrightarrow M^{r+1}/v$, so that (c) implies $\ker(f)=\im(v)$.
Similarly, (c) implies $\ker(v) = \im(f)$.

Finally we show that (c) implies (a).
We may assume that $M$ is concentrated in the finite interval $[a,b]$, and then (c) gives a
sequence of injections of finite-dimensional $k$-vector spaces
$$
\ldots \rightarrow 0 \rightarrow M^a/v \hookrightarrow M^{a+1}/v \hookrightarrow \ldots \hookrightarrow M^b/v = M^b \mathop{\rightarrow}\limits^{f}_{\sim} M^{b+1} \mathop{\rightarrow}\limits^{f}_{\sim} \ldots
$$
Note that $M^{r-1} \mathop{\leftarrow}\limits^{v} M^{r}$ is an isomorphism for $r\leq a$, and that
$M^r \mathop{\leftarrow}\limits^{v} M^{r+1}$ is zero for $r\geq b$, because the map $f$ in the other
direction is bijective, and $fv=p=0$.
Now take a $k$ basis $m^a_1,\ldots m^a_{d_a}$ of $M^a/v$, a $k$-basis $m^{a+1}_1,\ldots,m^{a+1}_{d_{a+1}}$
of $(M^{a+1}/v)/f(M^a/v)=M^{a+1}/(v,f)$ etc. up to a $k$-basis $m^b_1,\ldots $ of $M^b/(f,b)$, and lift these elements
to elements $\hat{m}^a_1,\ldots \hat{m}^a_{d_a}, \hat{m}^{a+1}_1,\ldots$ in $M^a$, $M^{a+1}$ etc.
Then the universal property of the free gauges $k(i)$ (see section \ref{ssec-1.3}) gives a morphism
of $k$-gauges
$$
g: F = \bigoplus_{i = a}^{b} k(-i)^{d_i} \longrightarrow M
$$
mapping the canonical elements of $k(-i)^i$ to the elements $\hat{m}^{i}_1,\ldots,\hat{m}^{i}_{d_i}$.
Since this map is surjective modulo $(f,v)$, it is surjective by lemma \ref{lem.nakayama}.
Moreover, by construction the map $F^b \rightarrow M^b/v=M^b$ is bijective: both spaces have
dimension $d = d_a + \ldots + d_b$. But, as remarked above, (c) implies that $M$ is quasi-rigid,
so that each $M^r$ has dimension $d$, and the same is true for each $F^r$. Therefore the surjective
map $g$ is an isomorphism, and we have shown (a).
\end{proof}

We draw the following consequences for $W(k)$-gauges for a perfect field $k$.

\begin{corollary}\label{cor.nakayama}
Let $k$ be a perfect field of characteristic $p>0$, and let $M$ be gauge of finite type
over $W=W(k)$. If $M/(p,f,v)M=0$, then $M=0$. Consequently, $M$ is generated by homogeneous
elements $m_1,\ldots,m_r$ if and only if their residue classes generate $M/(p,f,v)M$.
\end{corollary}

\begin{proof}
This follows from lemma \ref{lem.nakayama} by the usual Nakayama lemma for the local ring $W$,
because the components $M^s$ of $M$ are finitely generated $W$-modules.
\end{proof}

Next we characterize free $W$-gauges of finite type. Obviously, their components free $W$-modules.
But this condition does not suffice.

\begin{theorem}\label{thm.virtual.crystals.2}
Let $M$ be a $W$-gauge of finite type with free components. Then the following conditions
are equivalent.

\smallskip (a) $M$ is a free $W$-gauge.

\smallskip (b) $N=M/pM$ is a free $k$-gauge.

\smallskip (c) The map $M^r/v \mathop{\longrightarrow}\limits^{f} M^{r+1}/v$ is injective for all $r$.

\smallskip (d) The map $M^{r+1}/f \mathop{\longrightarrow}\limits^{v} M^r/f$ is injective for all $r$.

\end{theorem}

\begin{proof} (a) trivially implies (b), but (b) also implies (a):
Assume $N = M/pM$ is free, say isomorphic to $\oplus_i k(i)^{d_i}$.
By the universal property of free gauges (= free modules over $D(k)$ and $D(W)$, respectively) we
can lift the isomorphism modulo $p$ to a morphism $F = \oplus_i W(i)^{d_i} \rightarrow M$, which is
surjective by corollary \ref{cor.nakayama}. Since this map is an isomorphism modulo $p$, and all
components are free $W$-modules, it is an isomorphism.

Next we remark that the maps in (c) and (d) can be identified with the maps
$$
N^r/v \mathop{\longrightarrow}\limits^{f} N^{r+1}/v \qaq N^{r+1}/f \mathop{\longrightarrow}\limits^{v} N^r/f\,,
$$
respectively, because $pM$ is contained in both $fM$ and $vM$, by the equality $fv=vf=p$.

Therefore (b) is equivalent to the conjunction of (c) and (d), by theorem \ref{thm.virtual.crystals.1}.

But (c) and (d) are in fact equivalent in our situation: Assume (c).
To show the injectivity in (d) let $y\in M^{r+1}$ with $v(y) = f(x)$,
where $x \in M^{r-1}$. Then (c) implies $x = v(z)$ with $z\in M^r$. We get $v(y) = f(v(z)) = v(f(z)$,
and hence that $y = f(z)$, because $v$ is injective ($fv = p$, and $M$ is torsion-free as $W$-module).
A similar reasoning shows that also (d) implies (c), again since $M$ is a torsion-free $W$-module.
Therefore properties (a) to (d) are equivalent.
\end{proof}

\begin{corollary}\label{cor.virtual.crystals}
A $W$-gauge of finite type $M$ is free if and only if it comes from a virtual crystal over $k$,
i.e., is the underlying gauge of a $\varphi$-$W$-gauge in the essential image of the functor in
theorem \ref{thm.virtual.crystals.1}.
\end{corollary}

\begin{proof}
Assume that $M$ comes from the virtual $k$-crystal
$(D,\phi,L)$. Then $M$ has free components, and we show condition (c) in theorem \ref{thm.virtual.crystals.2}.
We have
$$
M^r = \{\,x\in L\,\mid\,\phi(x) \in p^rL\,\}\,,\quad f(x) = px\,, \quad\mbox{ and }\quad v(x) = x\,,
$$
by definition. Now let $x\in M^r$ with $f(x) = v(y)$ for $y\in M^{r+2}$. Then, by definition, we have
$x, y \in L$ satisfying $\phi(x) = p^rz$ with $z \in L$ and $\phi(y) = p^{r+2}t$ with $t\in L$,
and $px = y$. Then $p^{r+1}z = p\phi(x) = \phi(px) = \phi(y) = p^{r+2}t$ and hence $z=pt$,
since $L$ is torsion-free. This implies $\phi(x) = p^{r+1}t$, i.e., $x \in vM^{r+2}$.
Conversely we show that any free $W$-gauge arises from a virtual crystal.
By considering sums in both categories, we may consider the
case $M = W(i)$ for some $i\in \z$. But this gauge arises from the virtual crystal $(B,\phi,W)$
where $\phi(b) = p^{-i}\sigma(b)$.
\end{proof}

We can now strengthen this result and characterize the image of the functor in theorem \ref{thm.virtual.crystals.1}.

\begin{theorem}\label{thm.virtual.crystals.3}
A $\varphi$-$W$-gauge $(M,f,v,\varphi)$ of finite type comes from a virtual crystal over $k$, i.e., lies in the essential
image of the functor in theorem \ref{thm.virtual.crystals.1}, if and only if the underlying gauge is a free $W$-gauge.
\end{theorem}

\begin{proof}
One direction follows from corollary \ref{cor.virtual.crystals}. For the other direction assume that
$(M,f,v)$ is free. Assume that $M$ is concentrated in the finite interval $[a,b]$.
The functor in theorem \ref{thm.virtual.crystals.1} is compatible with twists: If $\phi$ is multiplied by $p^i$,
then the associated gauge $M$ is replaced by $M(-i)$. Therefore we may assume that $a=0$. Then the $\varphi$-$W$-gauge
corresponds to the finite diagram
$$
M^0 \mathop{\toto}\limits^{f}_{v} \ldots \mathop{\toto}\limits^{f}_{v} M^{r-1} \mathop{\toto}\limits^{f}_{v} M^r
\mathop{\toto}\limits^{f}_{v} M^{r+1} \mathop{\toto}\limits^{f}_{v} \ldots \mathop{\toto}\limits^{f}_{v} M^b\,,
$$
together with a $\sigma$-semi-linear isomorphism
$$
\varphi: M^{+\infty} = M^b \mathop{\longrightarrow}\limits^{\sim} M^0 = M^{-\infty}\,.
$$
All maps $f$ and $v$ are injective, because $fv=vf=p$. Let $L = M^0$, and define the $\sigma$-linear
endomorphism
$$
\phi = \varphi\circ f^b:\; L = M^0 \mathop{\longrightarrow}\limits^{f^b} M^b \mathop{\longrightarrow}\limits^{\varphi}_{\sim} M^0 = L\,.
$$
Then $(L,\varphi)$ is a crystal over $k$, and we claim that the associated $\varphi$-$W$-gauge
is canonically isomorphic to $(M,f,v,\varphi)$. In fact, first we claim that, for $0 \leq r \leq b$,
the injective map $v^r: M^r \rightarrow L$ has the image $L^r = \{\,x\in L\,\mid\, \phi(x)\in p^rL\,\}$.

First of all, we have $\phi(v^rx) = \varphi(f^bv^rx) = \varphi(p^rf^{b-r}x) \in p^rL$.
Conversely, by the assumption and criterion (c) in theorem \ref{thm.virtual.crystals.2}, all maps
$$
M^0/vM^1 \mathop{\longrightarrow}\limits^{f} M^1/vM^2 \mathop{\longrightarrow}\limits^{f} M^2/vM^3 \mathop{\longrightarrow}\limits^{f} \ldots
$$
are injective. If now $x\in L = M^0$ with $\phi(x) = \varphi(f^bx) = p^ry$ with $y\in L$, then
$f^bx = p^r\varphi^{-1}y = f^rv^r\varphi^{-1}y$ and hence $f^{b-r}x=v^r\varphi^{-1}y$ by injectivity of $f$.
By the sequence of injective maps above this implies inductively $x=x_1$ with $x_1\in M^1$,
hence $f^{b-r}vx_1 = v^r\varphi^{-1}y$, hence $f^{b-r}x_1 = v^{r-1}\varphi^{-1}y$ by injectivity of $v$ on $M$,
hence $x_1 = vx_2$ with $x_2 \in M^2$ etc. and inductively $x=v^rx_r$ with $x_r\in M^r$.
Thus we have $L^r = v^rM^r$ as claimed.

Identifying $L^r$ with $M^r$ via $v^r$ the maps $v$ become
inclusions, and the maps $f$ become multiplication by $p$, because $fv=p$.
Finally one sees that the map $\varphi: M^b \longrightarrow M^0$ identifies with
$p^{-b}\phi: L^b \longrightarrow L^0$, sending $y$ to $p^{-b}\phi(y)$ as in the
construction of the $\varphi$-gauge associated to $(L,\phi)$.
\end{proof}

\begin{corollary}\label{cor.virtual.crystals.final} The functor in theorem \ref{thm.virtual.crystals.1}
induces an equivalence of categories
$$
( \mbox{ virtual crystals } (D,\phi,M) \mbox{ over } k ) \longrightarrow
( \mbox{ finite type free } \varphi\mbox{-}W\mbox{-gauges }(M,f,v,\varphi) )\,,
$$
where we call a $\varphi$-$W$-gauge free if the underlying gauge is free.
\end{corollary}

\begin{remark}\label{rem.free.gauges}
Let $M$ be a $W$-gauge of finite type with free components. If $M$ is concentrated in a point,
i.e., in an interval $[a,a]$, then it is obviously free, viz., isomorphic to $W(-a)^d$ for some $d$.
If $M$ is concentrated in an interval of length 1, it is free as well. For this we use criterion (c)
from theorem \ref{thm.virtual.crystals.2}. If $M$ is represented by
$$
M^a \mathop{\toto}\limits^{f}_{v} M^{a+1}\,,
$$
we have only have to show the injectivity of
$f: M^a/vM^{a+1} \rightarrow M^{a+1}/vM^{a+2}$; the injectivity at the other places is clear.
But the image of $v: M^{a+2} \rightarrow M^{a+1}$ is $pM^{a+1}$, and if $x\in M^a$ and $fx = vy = pz = fvz$
with $z\in M^{a+1}$, then $x = vz$ by injectivity of $f$.

But already if $M$ is concentrated in an interval of length 2, $M$ is not in general free.
A counterexample is the $W$-gauge
$$
N \mathop{\toto}\limits^{p}_{\supset} pN \mathop{\toto}\limits^{p}_{=} pN\,,
$$
for any free $W$-module $N\neq 0$, since here $N/v = N/p \rightarrow pN/v = 0$ is not injective.
\end{remark}

\subsection{\it Gauges and Dieudonn\'{e} modules.}\label{ssec-2.3}

\medskip

{\it A Dieudonn\'{e}-module of finite type over $k$} is a
$W$-module of finite type $M$ endowed with two additive
endomorphisms $F$ and $V$ such that $$FV= VF=p\hbox{ and  }F(\lambda
x)=\sigma(\lambda)Fx\ , \ V(\sigma(\lambda)x)= \lambda Vx\
(\forall\lambda\in W,\
\forall x\in M)$$

With an obvious definition of morphisms, Dieudonn\'{e} modules of
finite type over $k$ form a $\mathbb Z_p$-linear abelian category $Dieud(k)$.

Let M be an object of ${\mathcal G}_{ft}^{[-1,0]}(W)$. To give $M$ is the
same as giving $(M^{-1},M^{0},f,v,\varphi)$ where $M^{-1}$ and $M^0$
are $W$-modules of finite type, $f:M^{-1}\to M^0$ and $v:M^{0}\to
M_{1}$ are $W$-linear maps such that $fv=pid_{M^{0}}$ and
$vf=pid_{M^{-1}}$ and $\varphi:M^0\to M^{-1}$ is a bijective
$\varphi$-linear map. We define $F,V:M^{-1}\to M^{-1}$ by $F=\varphi f$
and $V=v\varphi^{-1}$. This turns $M^{-1}$ into a Dieudonn\'{e} module
of finite type over $k$. In this way, we get a functor
$$
{\mathcal G}_{ft}^{[-1,0]}(W)\longrightarrow Dieud(k)\,,
$$
which is an equivalence of categories.

\smallskip
Dieudonn\'{e}-modules arise from $p$-divisible groups over $k$ or from
the first crystalline cohomology of smooth projective varieties $X$ over $k$.
By the theory of the de Rham-Witt complex, the $i$-th crystalline cohomology of $X$
gets the structure of what could be called a `Dieudonn\'{e}-module of weight $i$':
a finitely generated $W$-module $M$ together with a $\sigma$-linear endomorphism $F$ and
a $\sigma^{-1}$-linear endomorphism $V$ such that $FV = VF = p^i$. Such a structure can also be obtained
by a $\varphi$-gauge with free components which is concentrated in an interval of length $i$,
$$
M^a \toto \ldots \toto M^{r-1} \mathop{\toto}\limits^{f}_{v} M^r
\mathop{\toto}\limits^{f}_{v} M^{r+1} \toto \ldots \toto M^{a+i} \mathop{\longrightarrow}\limits^{\varphi}_{\sim} M^a\,,
$$
by letting $M = M^a$, $F = \varphi f^i$, and $V = v^i \varphi^{-1}$. This gives a functor
$$
{\mathcal G}_{ft}^{[a,a+i]}(W) \longrightarrow  Dieud^i(k)\,,
$$
where $Dieud^i(k)$ is the category of Dieudonn\'{e}-modules of weight $i$, whose morphisms
are linear maps compatible with $F$ and $V$. However, for $i>1$ this functor is no longer
an equivalence of categories, because it forgets all information concerning the modules
$M^{a+1}, \ldots, M^{a+i-1}$.

\smallskip
One aim of this paper is to establish a canonical cohomology theory giving $\varphi$-$W$-gauges
of finite type $H^i_g(X/W)^\bullet$ for each $i$, concentrated in the interval $[0,i]$, whose
associated Dieudonn\'{e}-module of weight $i$ is the $i$-th crystalline cohomology $H^i_{cris}(X/W)$.
This new cohomology theory thus refines the crystalline cohomology.

\subsection{\it Effective, coeffective modules and truncations.}\label{ssec-2.4}

\medskip
We say that a gauge $M$ over $W_n=W_n(k)$ (for $1\leq n \leq\infty$) is {\it effective} (resp.
{\it coeffective}) if $v:M^{r}\to M^{r-1}$ is an isomorphism for
$r\leq 0$ (resp. $f:M^r\to M^{r+1}$ is an isomorphism for $r\geq 0$).

For any object $M$ of $\cG_{ft}(W_n)$, $M(i)$ is effective for $i<<0$ and
coeffective for $i>>0$.

To any gauge $W_n$-gauge $M$, we may associate the
co-effective $W_n$-gauge $M_{\leq 0}$ defined as follows:
we have $M_{\leq 0}^r=M^r$ for $r<0$ and $M_{\leq 0}^{r}=M^0$ for
$r\in\mathbb N$, with
$fx=x$ if $x\in M_{\leq 0}^r$ and $r\geq 0$ and $vx=px$ if $x\in
M_{\leq 0}^r$ and $r>0$.

If $M$ is of finite type, so is $M_{\leq 0}$ and we may view
$M\mapsto M_{\leq 0}$ as a functor from $\mathcal \cG_{ft}(W_n)$ to the
full sub-category $\cG_{ft}^{\leq 0}(W_n)$ of coeffective gauges of finite type, which is a
right adjoint of the inclusion functor. We observe that the obvious
map $M_{\leq 0}\to M$ is not in general injective.

For any $W_n$-gauge $M$, the natural maps
$M^0\to (M_{\leq0})^{+\infty}$ and $(M_{\leq
0})^{-\infty}\to M^{-\infty}$ are
isomorphisms. Therefore, we may also view $M \mapsto M_{\leq 0}$ as a
functor from the category $\varphi$-$\cM_{ft}(W_n)$ of $\varphi$-$W_n$-modules of finite type
to the full sub-category $\varphi$-$\cM_{ft}^{\leq 0}(W_n)$ of
coeffective $\varphi$-$W_n$-modules, by defining
$\varphi:M_{\leq 0}^{+\infty}\to M_{\leq 0}^{-\infty}$ as the
compositum $\varphi^0$ of the natural map $M^{0}\to M^{+\infty}$ with
the original $\varphi$.

Again this functor is a right adjoint of the inclusion functor. We
observe that, when $M$ is a $\varphi$-gauge, $M_{\leq 0}$ is not always a $\varphi$-gauge.

\section{Zariski-gauges and $F$-zips over schemes of characteristic $p$}\label{sec-3}

Let $S$ be a scheme of characteristic $p$.

\subsection{\it Zariski-gauges and Zariski-$\varphi$-gauges}\label{ssec-3.1}

\medskip
The $\varphi$-ring associated to $S$ in the Zariski topology is defined as the commutative ring (in this topology)
$D(S) = \cO_S[f,v]/(fv)$ together with the ring morphism $\varphi: D(S)^{+\infty} = \cO_S \longrightarrow \cO_S = D(S)^{-\infty}$
which is given by the absolute Frobenius $F=F_S$ on $\cO_S$ (which is the identity on $S$ and the map $x \mapsto x^p$
on the sections. Note that this is only an isomorphism if $S$ is a perfect scheme, so $D(S)$ is not in general a perfect
$\varphi$-ring.

\medskip
A Zariski gauge $M$ over $S$ is a $D(S)$-module, which is obviously just a gauge in the category of $\cO$-modules.
It is called coherent if it is of finite presentation over $D(S)$. Hence, if $S$ is noetherian, it is coherent if
and only if $M$ is concentrated in a finite interval and each component $M^r$ is a coherent $\cO_S$-module.

\medskip
In accordance with the definitions in section \ref{ssec-1.4}, a Zariski $\varphi$-module $(M,\varphi)$
is a Zariski gauge $M$ together with an $\cO_S$-linear morphism
$$
\varphi: \,(M^{+\infty})^{(p)} = \cO_S\,{}_{Fr}\hspace{-1mm}\otimes_{\cO_S} M^{+\infty} \longrightarrow M^{-\infty}\,.
$$
Here $\cN^{(p)}=\cO_S\,{}_{Fr}\hspace{-1mm}\otimes_{\cO_S}\cN$ is the usual Frobenius twist (twist with the absolute Frobenius $Fr$)
of an $\cO_S$-module $\cN$. Since this operation is a right exact functor, the $\varphi$-$\cO_S$-modules form
an abelian category, see section \ref{ssec-1.4}. A $\varphi$-module is a $\varphi$-gauge if $\varphi$ is
an isomorphism. They form a subcategory which is closed under direct sums and direct factors.

\subsection{\it The relationship with $F$-zips}\label{ssec-3.2}

\medskip
An $F$-zip over $S$ is defined as a locally free coherent $\cO_S$-module $\cM$ together with
\begin{itemize}
\item[(a)] A descending filtration $C = (C^i)_{i \in \z}$ on $\cM^{(p)}$ by locally direct summands,

\item[(b)] An increasing filtration $D = (D_i)_{i \in \z}$ on $\cM$ by locally direct summands,

\item[(c)] a family $(\varphi_i)_{i\in\z}$ of $\cO_S$-linear isomorphisms
$\varphi_i:\, C^i/C^{i+1} \mathop{\longrightarrow}\limits^{\sim} D_i/D_{i-1}$.
\end{itemize}
A morphism of $F$-zips $(\cM,C,D) \rightarrow (\cM',C',D)$ is an $\cO_S$-linear
morphism $\cM \longrightarrow \cM'$ respecting the filtrations. (This is the modified
definition in \cite{Wed}, improving the original definition of Moonen and Wedhorn \cite{MW}.
If $S$ is perfect, then both definitions are equivalent.)

\medskip
Then one has the following result.

\begin{theorem}\label{thm.zips.and.gauges} (see \cite{Schn})
There is a canonical full embedding of categories
$$
(\, F\mbox{-zips over } S\,) \longrightarrow (\, \varphi\mbox{-}\cO_S\mbox{-gauges} \,)\,.
$$
The essential image consist of the $\varphi$-$\cO_S$-gauges which are rigid, coherent, and have locally free components.
\end{theorem}

\section{Zariski-gauges and displays over schemes of characteristic $p$.}\label{sec-4}

Langer and Zink defined the notion of a display over a ring $R$ of positive characteristic.
Let $W(R)$ be the ring of Witt vectors for $R$, and let $I(R) = VW(R)$, the image of the
Verschiebung $V$ on $W(R)$. Then a {\it predisplay} consists of the following data:

\medskip\noindent
1) A chain of morphisms of $W(R)$-modules
$$
\ldots \longrightarrow P_{i+1} \mathop{\longrightarrow}\limits^{\iota_i} P_i \longrightarrow \ldots
\longrightarrow P_1 \mathop{\longrightarrow}\limits^{\iota_0} P_0\,,
$$
2) For each $i \geq 0$ a $W(R)$-linear map
$$
\alpha_i:\; I_R\otimes_{W(R)} P_i \longrightarrow P_{i+1} \,,
$$
3) For each $i \geq 0$ a Frobenius-linear map
$$
F_i: \; P_i \longrightarrow P_0\,.
$$
These are required to satisfy the following conditions: The composition $\iota_i\circ\alpha_i$ is the
multiplication $I_R\otimes P_i \longrightarrow P_i$, and one has
$$
F_{i+1}(\alpha_i(V(\eta)\otimes x)) = \eta F_ix, \quad \mbox{  for  } \eta\in I_R, x \in P_i\,.
$$
Predisplays form an abelian category, in an obvious way.

Finally, a predisplay is called a {\it display} of degree $d$ if there are finitely generated
projective $W(R)$-modules $L_0, \ldots,L_d$ such that
$$
P_i = (I\otimes L_0) \oplus \ldots \oplus (I\otimes L_{i-1}) \oplus L_i \oplus \ldots \oplus L_d\,,
$$
and such that the structural maps $\iota_i, \alpha_i$ and $F_i$ come from Frobenius-linear maps
$$
\Phi_i: L_i \longrightarrow L_0 \oplus \ldots \oplus L_d
$$
with the property that $\oplus_i \Phi_i$ is a Frobenius-linear automorphism of $L_0\oplus\ldots\oplus L_d$.
(See \cite{LZ} Definition 2.5 for the precise prescription how to get a predisplay out of these data).
These data are not supposed part of the datum of a display, only the existence matters.
Thus the displays form a full subcategory of the category of predisplays.

Then one has the following result.

\begin{theorem}\label{thm.displays-gauges} (see \cite{Wid})
There is a fully faithful embedding of categories
$$
(\, \mbox{ predisplays over } R \,) \longrightarrow ( \, \varphi\mbox{-}W(R)\mbox{-modules} \, )\,,
$$
which maps the category of displays to the category of $\varphi$-$W(R)$-modules for which $\varphi$
is an epimorphism.
\end{theorem}

\section{Topologies.}\label{sec-5}

Recall that a morphism of schemes is said to be {\it syntomic} if it
is flat and locally a
complete intersection (which implies that it is locally of finite type).

We say that a morphism $\mbox{Spec } B\to \mbox{Spec } A$ of affine schemes of
characteristic $p$ is {\it
an extraction of $p$-th root} if one may write $B=A[t]/(t^p-a)$, for
some $a\in A$.

\renewcommand{\thefootnote}{\arabic{footnote}}

We say that a morphism $X\to Y$ of $\Bbb F_p$-schemes is a {\it
$p$-root-morphism} (resp. is {\it
quiet}\footnote{quiet abbreviates quasi-\'{e}tale}) if locally for the Zariski topology (resp. for the
\'{e}tale topology), it may
be written as a successive extractions of $p$-th roots.

In characteristic $p$, we have the following inclusions~:

$$
\begin{array}{cccccc}
\hbox{open immersions} & \subset&  \hbox{\'{e}tale morphisms}& \subset& \hbox{flat morphisms}\\
\cap & &\cap & & \cup\\
p\hbox{-morphisms} & \subset & \hbox{quiet morphisms} & \subset & \hbox{syntomic morphims}
\end{array}
$$

\bigskip
A ring $A$  of characteristic $p$ is {\it perfect} if the Frobenius
$\varphi: A \rightarrow A, a\mapsto a^p$ is bijective. A scheme
of characteristic $p$ is perfect if ${\mathcal O}_X$ is a sheaf of
perfect rings.

\begin{lemma}\label{lem.perfect}
Let $A$ be a noetherian ring of characteristic $p$. Then the following holds.

(1) If $\varphi: A \rightarrow A$ is surjective, then $A$ is perfect.

(2) The subring $A_{per} = \cap_n \varphi^n(A)$ is a perfect ring.

\end{lemma}

\begin{proof}
Consider the ideals $A(p^n) = \{\,a\in A\,\mid\,a^{p^n}=0\,\}$. Since $A$ is noetherian,
the ascending sequence $A(p) \subseteq A(p^2) \subseteq A(p^3)\ldots$ becomes stationary.
Assume that the union of all these ideals is equal to $A(p^N)$ for some $N\geq 1$, say.

(1): Assume that $\varphi$ is surjective and that $a \in A$ with $a^p=0$. Then there
exists an element $b\in A$ with $b^{p^N}=a$. Hence $b^{p^{N+1}}=a^p=0$ and thus
$0 = b^{p^N} = a$.

(2): Let $x\in A_{per}$. Then there exist elements $x_{N}, x_{N+1},x_{N+2},\ldots$
in $A$ such that
$$
x = x_{N}^{p^{N+1}} = x_{N+1}^{p^{N+2}} = x_{N+2}^{p^{N+3}} = \ldots
$$
This implies $(x_{N} - x_{N+1}^p)^{p^{N+1}} = 0$ and hence $(x_{N} - x_{N+1}^p)^{p^N} = 0$,
i.e., $x_{N}^{p^N} = (x_{N+1}^{p^N})^p$. Setting $y_i=x_i^{p^N}$ for $i\geq N$, we similarly get
$y_i = y_{i+1}^p$ for all $i\geq N$, so that $y_N \in A_{per}$ and $x = y_N^p$.
Hence $\varphi$ is surjective on $A_{per}$. If now $a \in A_{per}$ with $a^p=0$,
we have an element $b\in A_{per}$ with $a = b^{p^N}$, and we conclude as before,
arguing inside $A$, that $a=0$.
\end{proof}

As a consequence, if $X$ is a locally
noetherian scheme of characteristic $p$, there is a unique morphism $X\to X_{per}$
with a perfect scheme $X_{per}$
such that any morphism $X\to Y$ with $Y$ perfect factors
uniquely through $X_{per}$.
We say that a scheme $X$ of characteristic $p$ is {\it absolutely
syntomic} if it is
locally noetherian and if the morphism $X\to X_{per}$ is
syntomic.

(We do not know wether or not $X$ locally noetherian implies $X_{per}$ locally
noetherian, but we do not care).

A field $k$ of characteristic $p$ is absolutely syntomic if and only
if the extension $k/\varphi(k)$ is finite. If $Y$ is absolutely
syntomic and if $X\to Y$
is syntomic, $X$ is absolutely syntomic.

We denote by ${\mathcal C}$ the full subcategory of the category of schemes of
characteristic $p$ whose objects are absolutely syntomic schemes.

In this paper,  an {\it admissible class of morphisms of ${\mathcal C}$} is a
class of syntomic morphisms of ${\mathcal C}$ containing all the open
immersions and stable  under composition and base change. A {\it
$p$-admissible class of morphisms of ${\mathcal C}$} is an admissible
class containing the extractions of $p$-th roots.

For any admissible class $E$,  call ${\mathcal C}_{E}$ the site whose
underlying category is ${\mathcal C}$, with surjective
families of $E$-morphisms (that is of morphisms of ${\mathcal C}$
belonging to $E$) as coverings.

We set $E=$Zar (resp. $p$, \'{e}t, quiet, synt) for the class of open
immersions (resp. $p$-morphisms, \'{e}tale morphisms, quiet morphisms,
syntomic morphisms).

Let $X$ be any object of ${\mathcal C}$. For any admissible class $E$ of
morphisms of ${\mathcal C}$, we call ${\mathcal C}_{X,E}$ the site
whose underlying category is the full subcategory of $X$-schemes
$Y\to X$ such that $Y$ is
an object of ${\mathcal C}$, with surjective families of $E$-morphisms as covering.
The corresponding sheaves shall be called {\it sheaves
(over ${\mathcal C}$) for
the $E$-topology}.

If moreover $X$ is noetherian, we call $X_{quiet}$ the site whose
underlying category is the
full subcategory of $X$-schemes $Y\to X$ such that the structural
morphism is quiet of
finite type, with finite surjective families of quiet morphisms as
coverings.

\begin{remarks}: Let $k$ be a field of characteristic $p$ such that the
extension $k/\varphi(k)$ is finite.

 (1) If $k$ is not perfect, $k_{quiet}=k_{quiet}$ is {\it the smallest
Grothendieck topology able to deal with
all finite extensions of $k$}. More precisely:

i) any finite extension of $k$ is a quiet $k$-algebra,

ii) if $X\to k$ is a quiet morphism of finite type, there is a
surjective quiet morphism of
finite type $U\to X$, such that $U=\mbox{Spec } (k_1\otimes_k\otimes k_2\otimes_k
\ldots\otimes_{k}k_d)$ with $k_1,k_2,\ldots,k_d$ finite fields
extensions of $k$.

 (2) If $k$ is perfect, the functor which associates to any finite
commutative group scheme
over $k$ the sheaf it defines on $k_{quiet}$ induces an equivalence
of categories between the category of
finite and flat commutative group schemes over $k$ and the category of
abelian groups
over $k_{quiet}$ which are representable. This is due to the fact
that any finite
commutative group scheme over $k$ is quiet. Observe that we have a
similar statement for
fields of characteristic $0$ and the \'{e}tale topology.
\end{remarks}

\section{The rings $\mathcal O^{cris}_n$ and the $\varphi$-rings ${\mathcal G}_n$.}\label{sec-6}

We continue to call a ring object in a topos $\cT$ simply a ring in $\cT$, or a ring over $\cC$ with
respect to the topology $E$ if the topos given by the category $\cC$ and the topology $E$.

\subsection{\it The $p$-adic ring $\mathcal O^{cris} = (\cO_n)_n$.}\label{ssec-6.1}

\medskip
A {\it $p$-adic ring} $R$ in a fixed topos ${\mathcal T}$ consists of
giving, for each $n\in\mathbb N$ a
commutative ring $R_n$ in ${\mathcal T}$, together with an isomorphism
$R_{n+1}/p^n\simeq R_n$. A $p$-adic ring $R$ in $\cT$ is {\it flat} if $R_n$ is
flat over $\mathbb Z/p^n$, for all $n\in\mathbb N$, i.e., if the sequence
$$R_n\,\smash{\mathop{\hbox to 4mm{\rightarrowfill}}
\limits^{\scriptstyle p}}\,R_n\,\smash{\mathop{\hbox to 4mm{\rightarrowfill}}
\limits^{\scriptstyle p^{n-1}}}\,R_n$$
is exact.
In this case, for $m,n\in\mathbb N$, we have exact sequences
$$0\to R_n \to R_{n+m}\,\smash{\mathop{\hbox to 7mm{\rightarrowfill}}
\limits^{\scriptstyle p^n}}\, R_{n+m}\to R_n\to 0\,,
$$
and, in particular, exact sequences
$$
0 \to R_n \to R_{n+m} \to R_m \to 0\,.
$$
Let $\cT = (\cC,E)$ be a ringed topos with $\mathcal C$ as in section \ref{sec-5}, a topology $E$,
and the structural sheaf of rings ${\mathcal O}$, defined by $\cO(X) = \cO_X(X)$ for a scheme $X$ in $\mathcal C$.
For $n\in\mathbb N$, a $\mathbb Z/p^n$
{\it -divided power thickening of} ${\mathcal O}$ {\it (for the $E$-topology)} is a triple $({\mathcal G},\theta,\gamma)$
where  ${\mathcal G}$ is a $\mathbb Z/p^n$-ring on
${\mathcal C}_E$, $\theta:{\mathcal G}\to {\mathcal O}$ is an epimorphism of
rings and $\gamma$
is a divided power structure on the kernel of $\theta$ such that, for
any object $U$ of ${\mathcal
C}$, any $x\in {\mathcal G}(X)$ and any $m\in\mathbb N$, we have
$\gamma_m(px)=(p^m/m!)x^m$.

The $\mathbb Z/p^n$-divided power thickenings of ${\mathcal O}$ for the
$E$-topology form, in an
obvious way, a category. For $E=p$, this category has an initial
object that we call $\mathcal O^{cris}_n$.
This can be shown

- either by working on the crystalline site and showing that
$$X\mapsto \mathcal O^{cris}_n(X):=H^{0}((X\to\mbox{Spec } W_n({\mathcal
O}_{X_{per}}))_{crys},\hbox{structural sheaf})$$
is a solution of the universal problem,

- or by constructing $\mathcal O^{cris}_n$ directly as the syntomic sheaf associated to the presheaf
$$
X \mapsto W_n^{DP}(X)\,,
$$
the divided power envelope of the ring of Witt vectors of $X$ (see \cite{FM}).

\bigskip
Moreover, for any admissible class $E$,  $\mathcal O^{cris}_n$ is also a sheaf for
the $E$-topology. Therefore,
if $E$ is $p$-admissible,
$\mathcal O^{cris}_n$ is also an initial object of the category of $\mathbb Z/p^n\mathbb Z$-divided
power thickenings
of ${\mathcal O}$ for the  $E$-topology.

Under the same assumption on $E$, the natural morphism ${\mathcal
O}_{n+1}^{cris}/p^n\to\mathcal O^{cris}_n$ is an isomorphism and the  $p$-adic
Ring $\mathcal O^{cris}=(\mathcal O^{cris}_n)_{n\in\mathbb N}$ is flat \cite{FM}, i.e., we have
natural exact sequences for all $n$ and $m$
$$
0 \longrightarrow \cO^{cris}_n \longrightarrow \cO^{cris}_{n+m} \longrightarrow \cO^{cris}_m \longrightarrow 0\,.
$$
The Frobenius $\varphi: a\mapsto a^p$ is an endomorphism of the structural
Ring ${\mathcal O}$.
By functoriality, it induces an endomorphism of $\mathcal O^{cris}_n$, that we still denote
by $\varphi$. This is also an endomorphism of ${\mathcal O}^{cris}$, i.e.
the projection ${\mathcal O}_{n+1}^{cris}\to\mathcal O^{cris}_n$ commutes with $\varphi$.

If $A$ is a perfect (noetherian) ring of characteristic $p$, we have
$\mathcal O^{cris}_n(A)=W_n(A)$, with the usual Frobenius.

\subsection{\it The $\varphi$-Ring ${\mathcal G}$.}\label{ssec-6.2}

\medskip
For each $n\in\mathbb N$, we want to define a $\varphi$-ring ${\mathcal G}_n$,
such that the $\cG_n$ form a $p$-adic $\varphi$-ring $\cG = (\cG_n)_{n\in \n}$.

Morally, we get it
by the standard construction introduced in section \ref{ssec-2.2}, from the $p$-adic
ring $\cO^{cris} = (\cO^{cris}_n)$ and the Frobenius $\varphi: \cO^{cris} \rightarrow \cO^{cris}$ on it,
by defining
$$
\cG^r = \ker(\cO^{cris} \mathop{\longrightarrow}\limits^{\varphi} \cO^{cris} \longrightarrow \cO^{cris}/p^r) =
``\{\, x \in \cO^{cris}\,\mid\, \varphi(x) \in p^r\cO^{cris}\,\}\mbox{''}\,.
$$
This makes perfect sense in the setting of pro-objects, and then gives rise to the objects $\cG^r_n = \cG^r/{p^n}$
which are essentially constant pro-objects.

For a more elementary and direct approach we proceed as follows.
For all $n$ we set ${\mathcal G}_n^{0}=\mathcal O^{cris}_n$, and the sub-ring $\oplus_{r\leq 0}{\mathcal G}_n^r$
is the ring of polynomials in an indeterminate, called $v$, with coefficients in $\mathcal O^{cris}_n$ and $v$ is in degree $-1$.
In other words, for
  any object $U$ of ${\mathcal C}$ and any integer $r\leq 0$, ${\mathcal G}_n^r(U)$ is the free
  $\mathcal O^{cris}_n(U)$-module of rank one generated by $v^{-r}$.

  If $r \geq 0$ and $m\in \n$ with $m\geq r$, set
$$
\hat{\mathcal G}^r_m = \ker ({\mathcal O}_{m}^{cris}\,\smash{\mathop{\hbox to 4mm{\rightarrowfill}}
  \limits^{\scriptstyle \varphi}}\,{\mathcal O}_{m}^{cris}\,\smash{\mathop{\hbox to 4mm{\rightarrowfill}}
  \limits^{\scriptstyle proj}}\,{\mathcal O}^{cris}_m/p^r = {\mathcal O}_{r}^{cris})\,,
$$
so this is the sub-sheaf whose
  sections $x$ are such that $\varphi(x)$ is locally (for the
$p$-topology) divisible by
  $p^r$. For any $m\geq n+r$ we define
$$
{\mathcal G}_n^r = \hG^r_m/p^n\,.
$$
It is easily seen that this definition is independent of the
choice of $m \geq n+r$, and that this definition agrees with the previous definition
  ${\mathcal G}_n^0=\mathcal O^{cris}_n$ for $r=0$.

  If $U$ is an object of ${\mathcal C}$, if $m,r\in\mathbb N$ with $m\geq r$ and
if $x\in \hG^r_m(U)$, $y\in \hG^s_m(U)$, then $xy\in \hG^{r+s}_m(U)$. By going to the
  quotient this gives a map ${\mathcal G}_{n}^r\times {\mathcal G}_n^{s}\to
{\mathcal G}_{n}^{r+s}$
  which define the multiplication on the sub-ring $\oplus_{n\in\mathbb N}{\mathcal G}_n^{r}$.

  To complete the definition of the multiplication on ${\mathcal G}_n$, it
is enough to define
  the multiplication by $v:{\mathcal G}_{n}^r\to {\mathcal G}_n^{r-1}$ for
$r\geq 1$. It is induced by the inclusion $\hG^r_m \subset \hG^{r-1}_m$,
  via passing to the quotients modulo $p^n$.
  In this way ${\mathcal G}_n$ becomes a graded ring.

  To get the structure of $\varphi$-ring, we have to introduce $f,v$
  and $\varphi$.
   We have already defined the global section
$v\in{\mathcal G_n}^{-1}(\Bbb F_p)$. For all $m\in\mathbb N$, we have ${\mathcal
O}_m^{cris}(\Bbb F_p)=\mathbb Z/p^m\mathbb Z$. For $n\in\mathbb N$ and $m\geq n+1$, the
image of $p$ belongs to
${\mathcal G}_{m,1}(\Bbb F_p)$ and we call $f$ its image in ${\mathcal G}_n^1(\Bbb F_p)$. Observe that
$fv=p$.

The natural morphism $\mathcal O^{cris}_n={\mathcal G}_n^0\to {\mathcal G}_n^{-\infty}$ is an
isomorphism of rings
and we use it to identify ${\mathcal G}_n^{-\infty}$ with $\mathcal O^{cris}_n$.

Finally we define $\varphi$. Let $n,r\in\mathbb N$, let $U$ be an object of ${\mathcal C}$ and $x\in{\mathcal G}_n^{r}(U)$. If we
choose an integer $m\geq n+r$, we may find a covering $V$ of $U$, a
lifting $y$ of $x$ in
$\hG^r_m(V)$, a covering $W$ of $V$ and $z\in {\mathcal O}_m^{cris}(W)$ such that
$\varphi(y)=p^rz$. The image $\varphi_r(x)$ of $z$ in $\mathcal O^{cris}_n(W)$ is
independent of the
choices made and belongs to $\mathcal O^{cris}_n(U)$. In this way we have $\varphi_r: \cG^r_n \rightarrow \cO^{cris}_n$
such that ``\,$\varphi_r(x) = \varphi(x)/p^r$\,''. Obviously, we have
$\varphi_{r+1}(fx)=\varphi_r(x)$, therefore the morphisms $\varphi_r$ define a map
$$
\varphi:\lim\limits_{\stackrel{\textstyle \longrightarrow}{r\in\mathbb N,f}}{\mathcal G}_n^r={\mathcal G}_n^{\infty}\to
\mathcal O^{cris}_n={\mathcal G}_n^{-\infty}
$$
One sees that $\varphi$ is a morphism of rings; this corresponds to the equality
$$
``\,\varphi(x)/p^r\cdot\varphi(y)/p^s = \varphi(xy)/p^{r+s}\,\mbox{''}
$$
for sections $x\in \hG^r_m$ and $y \in \hG^s_m$ and the ring endomorphism $\varphi$ of $\cO^{cris}_m$.

\begin{theorem}\label{thm.phi.onD.iso}
For each $n\geq 1$, the map
$$\varphi:{\mathcal G}_n^{+\infty}\to {\mathcal G}_n^{-\infty}$$
is an isomorphism of rings.
\end{theorem}

In other words, ${\mathcal G}_n$ is a perfect $\varphi$-ring.
The proof of this theorem is given in sections \ref{ssec-6.3}--\ref{ssec-6.5} below.

Here we note the following properties.

\begin{lemma}\label{lem.G.rigid.and.flat}
The following holds for the gauges $\cG_n$.

(i) One has natural isomorphisms $\cG_{n+1}/p^n \to \cG_n$ for all $n$, i.e., $(\cG_n)_n$
is a $p$-adic ring.

(ii) The $p$-adic ring $\cG = (\cG_n)_{n\in\n}$ is flat.

(ii) The gauge $\cG_1$ is rigid.
\end{lemma}

\begin{proof}
Morally, all statements are proved in the same way as for the standard construction in
section \ref{sec-2}, and could be proved by noting that $\cO^{cris}$ is a flat (= ``torsion-free'')
$p$-adic ring. But for concreteness we give a proof by the above explicit construction.
We argue with local sections, i.e., all statements hold after possible passing to some cover.
Moreover, we constantly use that the flatness of $\cO^{cris}$ implies the exactness of the sequence
$$
\cO^{cris}_m \mathop{\longrightarrow}\limits^{p^{m-i}} \cO^{cris}_m \mathop{\longrightarrow}\limits^{p^i} \cO^{cris}_m
$$
for all $i\leq m$. Recall that $\cG^r_n = \hG^r_m/p^n$ for some fixed $m >> 0$ ($m \geq n+r$ suffices),
where $\hG^r_m = \{x\in \cO^{cris}_m\,\mid\, \varphi(x) \in p^r\cO^{cris}_m\}$.
Unless specified explicitly, we calculate inside $\cO^{cris}_m$.

\smallskip\noindent
(i):  This is trivial: $(\hG^r_m/p^{n+1})/p^n = \hG^r_m/p^n$.

\smallskip\noindent
(ii): For each $i\leq n$ and $m\geq n+r$ we show the exactness of
$$
\hG^r_m/p^n \mathop{\longrightarrow}\limits^{p^{n-i}}\hG^r_m/p^n \mathop{\longrightarrow}\limits^{p^i} \hG^r_m/p^n
$$
as follows: If $x \in \hG^r_m$ and $p^ix = p^ny$ with $y \in \hG^r_m$, then we have $p^i(x-p^{n-i}y)=0$
and hence $x-p^{n-i}y = p^{m-i}z$ with $z\in \cO^{cris}_m$ by flatness of $\cO^{cris}$.
Thus $x = p^{n-i}(y + p^{m-n}z)$ with $y, p^{m-n}z \in \hG^r_m$ (note that $m-n \geq r$).

\smallskip\noindent
(iii): Recall that $v$ and $f$ are induced by the inclusion $\hG^r_m \hookrightarrow \hG^{r-1}_m$
and the $p$-multiplication $\hG^r_m \mathop{\rightarrow}\limits^{p} \hG^r_m$, respectively.
First we show quasi-rigidity. If $x \in \hG^r_m$ and $x = py$ for $y \in \hG^{r-1}_m$, then $[x] = f[y]$ for the class
$[x]$ of $x$ in $\cG^r_1$ and the class $[y]$ of $y$ in $\cG^{r-1}_1$. On the other hand, if $x\in \hG^r_m$
with $px = pz$ for $z\in \hG^{r+1}_m$, then $p(x-z)=0$ and hence $x-z = p^{m-1}t$ with $t\in \cO^{cris}_m$.
Thus $x = z + p^{m-1}t \in \hG^{r+1}_1$  for $m\geq r+2$, i.e., $[x] = v[x]$. Now we show strictness.
Let $x \in \hG^r_m$ with $f[x] = 0 = v[x]$. This means that $x = py$ with $y\in \hG^{r-1}_m$ and $px = pz$
with $z\in\hG^{r+1}_m$. This implies $pz = px = p^2y$, hence $p^2\varphi(y) = p\varphi(z) = p^{r+2}t$
for some $t\in \cO^{cris}_m$. This implies $p^2(\varphi(y) - p^rt)= 0$ and thus $\varphi(y) - p^rt= p^{m-2}u$
for some $u\in \cO^{cris}_m$ by flatness of $\cO^{cris}$. We conclude $y\in \hG^r_m$ for $m\geq r+2$ and therefore
$[x] = [py] = 0$.
\end{proof}

\bigskip
\subsection{\it The structure of a (generalized) $F$-zip on $\mathcal O^{cris}_1$.}\label{ssec-6.3}

\medskip
Let ${\mathcal T} = (\mathcal C,E,\cO)$ be a ringed topos with the category $\cC$, an admissible topology $E$,
and the structure ring ${\mathcal O}$ as in section \ref{ssec-6.1}.

An effective generalized $F$-zip (over ${\mathcal O}$) is an
${\mathcal O}$-module $\mathcal M$ together with

-- a decreasing filtration $(F^r\mathcal M)_{r\in\mathbb N}$ by sub-${\mathcal
O}$-modules indexed by $\mathbb N$ such
that $F^0\mathcal M=\mathcal M$ and
$\cap_{r\in\mathbb N}F^r\mathcal M=0$,

-- an increasing filtration $(F_{r}\mathcal M)_{r\geq -1}$ by sub-${\mathcal
O}$-modules indexed by the
natural
integers $r\geq -1$ such that
$F_{-1}\mathcal M=0$
and $\cup_{r\in\mathbb N}F_r\mathcal M=\mathcal M$,

-- for all $r\in\mathbb N$, an isomorphism
$$
\varphi_r:F^{r}\mathcal M/F^{r+1}\mathcal M \mathop{\longrightarrow}\limits^{\sim} F_r\mathcal M/F_{r-1}\mathcal M
$$
of abelian sheaves which is semi-linear with respect to the absolute Frobenius.

\begin{remark}\label{rem.sigma-iso} Since the absolute Frobenius $Fr$ is an epimorphism on $\cO$,
one easily sees that this is equivalent to the fact that the associated morphism
$$
(F^{r}\mathcal M/F^{r+1}\mathcal M)^{(p)} = \cO\,{}_{Fr}\hspace{-1mm}\otimes_\cO (F^{r}\mathcal M/F^{r+1}\mathcal M) \longrightarrow F_r\mathcal M/F_{r-1}\mathcal M
$$
is an isomorphism. This ties this definition with the definition in section \ref{ssec-3.2}.
\end{remark}

\bigskip
We shall now define such a structure on $\mathcal O^{cris}_1$ (with ${\mathcal O}$ the
structural sheaf). We define $F_r = F_r\cO^{cris}_1 = im(\varphi_r)$ for $\varphi_r: \cG^r_1 \rightarrow \cO^{cris}_1$
and $F^r = F^r\cO^{cris}_1 = im(v^r)$ for $v^r: \cG^r_1 \rightarrow \cO^{cris}_1$.

\begin{proposition}\label{prop.4-term.sequences}
There are canonical exact sequences for all $r\geq 1$
$$
0 \longrightarrow F_r \longrightarrow \cG^{r+1}_1 \mathop{\longrightarrow}\limits^{v} \cG^r_1 \mathop{\longrightarrow}\limits^{\varphi_r} F_r \longrightarrow 0
$$
$$
0 \longrightarrow F^r \longrightarrow \cG^{r-1}_1 \mathop{\longrightarrow}\limits^{f} \cG^r_1 \mathop{\longrightarrow}\limits^{v^r} F^r
\longrightarrow 0\,.
$$
\end{proposition}

\begin{proof}
For the first sequence we claim that there is an exact sequence
$$
0 \rightarrow \hG^{r+1}_{r+2} \hookrightarrow \hG^r_{r+2} \mathop{\rightarrow}\limits^{\varphi_r} \cO^{cris}_1\,,
$$
where the morphism on the right is the composition $\hG^r_{r+2} \twoheadrightarrow \cG^r_1 \mathop{\rightarrow}\limits^{\varphi_r} \cO^{cris}_1$.
In fact, if $x \in \hG^r_{r+2}$, so that $\varphi(x) = p^ry$ with $y\in \cO^{cris}_{r+2}$, then
$\varphi_r([x]) = [y] \in \cO^{cris}_{r+2}/p= \cO^{cris}_1$ for the corresponding classes modulo $p$.
If $\varphi_r([x])=0$, then $y = py'$ with $y' \in \cO^{cris}_{r+2}$. This implies $\varphi(x) = p^{r+1}y'$
and hence $x \in \hG^{r+1}_{r+2}$. Applying now the snake lemma to the multiplication by $p$ on the
exact sequence
$$
0 \rightarrow \hG^{r+1}_{r+2} \hookrightarrow \hG^r_{r+2} \mathop{\rightarrow}\limits^{\varphi_r} F_r \rightarrow 0\,,
$$
we get an exact sequence
$$
\hG^r_{r+2}[p] \rightarrow  F_r \rightarrow \cG^{r+1}_1 \hookrightarrow \cG^r_1 \mathop{\rightarrow}\limits^{\varphi_r} F_r \rightarrow 0\,,
$$
where $A[p]$ means the $p$-torsion subsheaf of a sheaf $A$. But the map on the left is zero:
If $x\in \hG^r_{r+2}$ with $px=0$, then $x = p^{r+1}u$ for some $u \in \cO^{cris}_{r+2}$ by flatness of $\cO^{cris}$.
Then $\varphi(x) = p^{r+1}\varphi(u)$, and by definition, $\varphi_r([x]) = [p\varphi(u)] = 0 \in \cO^{cris}_1$.

\smallskip
For the second sequence we claim that we have an exact sequence
$$
p\hG^{r-1}_{r+1} \hookrightarrow \hG^r_{r+1} \mathop{\rightarrow}\limits^{v^r} \cO^{cris}_1\,,
$$
where the morphism on the right is the composition $\hG^r_{r+1} \twoheadrightarrow \cG^r_1 \mathop{\rightarrow}\limits^{v^r} \cO^{cris}_1$.
In fact, if $x\in \hG^r_{r+1}$, so that $\varphi(x) = p^ry$ with $y\in \cO^{cris}_{r+1}$, and $x = pz$ with $z \in \cO^{cris}_{r+1}$,
then $p^ry = \varphi(x) = p\varphi(z)$, hence $p(\varphi(z)-p^{r-1}y)=0$, so that $\varphi(z)-p^{r-1}y = p^rt$
with $t\in \cO^{cris}_{r+1}$. This implies $z \in \hG^{r-1}_{r+1}$, hence the claim, because $x = pz$.
Applying the snake lemma to the multiplication by $p$ to the exact sequence
$$
0 \rightarrow p\hG^{r-1}_{r+1} \rightarrow \hG^r_{r+1} \mathop{\rightarrow}\limits^{v^r} F^r \rightarrow 0\,,
$$
we get an exact sequence
$$
\hG^r_{r+1}[p] \rightarrow F^r \rightarrow  (p\hG^{r-1}_{r+1})/p \rightarrow \cG^r_1 \mathop{\rightarrow}\limits^{v^r} F^r \rightarrow 0\,.
$$

Now we claim that the first map is the zero map. In fact, if $x \in \hG^r_{r+1}$ with $px=0$, then $x = p^ry$ with $y\in \cO^{cris}_{r+1}$ by
flatness of $\cO^{cris}$. Hence $x$ is mapped to zero in $\cO^{cris}_1$.

Moreover we claim that the exact sequence
$$
0 \rightarrow \hG^{r-1}_{r+1}[p] \rightarrow \hG^{r-1}_{r+1} \mathop{\rightarrow}\limits^{p} p\hG^{r-1}_{r+1} \rightarrow 0
$$
taken modulo $p$ induces an isomorphism
$$
\cG^{r-1}_1 = \hG^{r-1}_{r+1}/p \mathop{\rightarrow}\limits^{\sim} (p\hG^{r-1}_{r+1})/p\,.
$$
In fact, for this it suffices to show that the induced morphism $(\hG^{r-1}_{r+1}[p])/p \rightarrow (\hG^{r-1}_{r+1})/p$
is zero. But if $x \in \hG^{r-1}_{r+1}[p]$, then $px=0$, so that $x = p^rt$ with $t\in \cO^{cris}_{r+1}$ by
flatness of $\cO^{cris}$. Then $x = py$ with $y=p^{r-1}t \in \cG^{r-1}_{r+1}$ as claimed.

Both claims together imply the second exact sequence in the proposition.
\end{proof}

\begin{proposition}\label{prop.cartier.iso.2}(Cartier isomorphism)
The subsheaves $F_r= F_r\cO^{cris}_1$ form an increasing filtration of $\cO^{cris}_1$ (i.e., $F_r \subseteq F_{r+1}$),
and the subsheaves $F^r = F^r\cO^{cris}_1$ form a decreasing filtration of $\cO^{cris}_1$ (i.e., $F^{r+1} \subseteq F^r$).
For each $r\geq 0$, the morphism $\varphi_r: F^r = \cG^r_1 \rightarrow im(\varphi_r)=F_r$ induces an isomorphism
$$
\overline{\varphi}_r: \,F^r/F^{r+1} \mathop{\longrightarrow}\limits^{\sim} F_r/F_{r-1}\,.
$$
\end{proposition}

\begin{proof}
The first two claims are clear. For the third look at the diagram with exact rows
$$
\begin{matrix}
\cG^{r+1}_1 & \mathop{\longrightarrow}\limits^{v} & \cG^r_1 & \mathop{\longrightarrow}\limits^{\varphi_r} & F_r & \longrightarrow & 0 \\
            &                                     &    \|   &                                             &     &                 &   \\
\cG^{r-1}_1 & \mathop{\longrightarrow}\limits^{f} & \cG^r_1 &  \mathop{\longrightarrow}\limits^{v^r}      & F^r & \longrightarrow & 0
\end{matrix}\,.
$$
Since $\varphi_rf=\varphi_{r-1}$ the diagram induces an epimorphism
$$
\varphi'_r: \,F^r \longrightarrow F_r/F_{r-1}\,.
$$
Explicitly: If $x \in F^r$, and $x=v^ry$ for $y\in \cG^r_1$, let $\varphi'_r(x)$ be the class of $\varphi_r(y)$,
which is well-defined modulo $F_{r^-1}$, because the kernel of $v^r$ is the image of $f$.
The kernel of this morphism is $F^{r+1}$: If $\varphi'_r(x) = \varphi_r(y) = \varphi_{r-1}(z) = \varphi_r(fz)$,
then $\varphi_r(y-fz)=0$, so that $y-fz = vt$ with $t\in \cG^{r+1}_1$ by the exactness of the upper row.
But then $x = v^ry = v^rfz + v^{r+1}t = v^{r+1}t\in F^{r+1}$, because $vf=0$.
\end{proof}

From Proposition \ref{prop.4-term.sequences} we already obtain one half of Theorem \ref{thm.phi.onD.iso}.

\begin{corollary}\label{cor.varphi.injective}
The morphism $\varphi: \cG^{+\infty} \longrightarrow \cG^{-\infty}$ is injective.
\end{corollary}

\begin{proof}
Let $x\in \cG^{+\infty}$ with $\varphi(x)=0$. Suppose $x$ is the image of an element $x_r$ under
$f_r: \cG^r_1 \rightarrow \cG^{+\infty}$. Then $\varphi_r(x_r)=0$. By the first sequence in
\ref{prop.4-term.sequences}, we have $x_r=vx_{r+1}$. But then $fx_r = fvx_{r+1}=px_{r+1}= 0$,
which implies $x=f_rx_r = 0$.
\end{proof}

To obtain the fact that $\cup_r F_r = \cO^{cris}$
(which gives the surjectivity of $\varphi$ above and hence the second half of Theorem \ref{thm.phi.onD.iso}),
and that $\cap_r F^r = 0$, we need some explicit calculations.

\subsection{\it Some calculations for the universal $F$-zip}\label{ssec.6-4}

\bigskip
We first observe that we have a morphism of rings
$$
f : {\mathcal O}\to \mathcal O^{cris}_1\,.
$$
In fact, if $X$ is an object of ${\mathcal C}$ and if $a\in{\mathcal
O}(X)$, one may find a quiet covering $Y\to X$ such that $a$ is the
image of some  $b\in\mathcal O^{cris}_1(Y)$ and $f(a)=b^{p}$ is independent of the
choice of $b$ and belongs to
$\mathcal O^{cris}_1(X)$.

The kernel $\tilde{F}^1\mathcal O^{cris}_1$ (often called $J_1^{[1]}$) of the canonical
epimorphism ${\mathcal
O}_1^{cris}\to {\mathcal O}$ is a dived power Ideal and, for all
$r\in\mathbb N$, we set $\tilde{F}^r{\mathcal O}_1^{cris}=J_1^{[r]}$, the $r$-th
divided power of $\tilde{F}^1\mathcal O^{cris}_1=J_1^{[1]}$. We set
$I=J_1^{[1]}/J_1^{[2]}$
and call it {\it the cotangent space}. We denote
$G^r\mathcal O^{cris}_1$ the abelian sheaf
$\tilde{F}^r\mathcal O^{cris}_1/\tilde{F}^{r+1}\mathcal O^{cris}_1=J_1^{[r]}/J_1^{[r+1]}$ (hence $G^0\mathcal O^{cris}_1={\mathcal O}$
and $G^{1}\mathcal O^{cris}_1=I$). On each $G^r\mathcal O^{cris}_1$, we have two different structures
of ${\mathcal O}$-Modules :

-- the {\it naive structure} which comes from the fact that
$J_1^{[r+1]}\subset J_1^{[r]}$ are  sub-${\mathcal O}$-modules of $\mathcal O^{cris}_1$,

-- the {\it nice structure} which comes from the fact that
$J_1^{[r+1]}\subset J_1^{[r]}$ are  sub-$\mathcal O^{cris}_1$-modules of $\mathcal O^{cris}_1$, but
that $J_1^{[1]}.J_1^{[r]}\subset J_1^{[r+1]}$ and $\mathcal O^{cris}_1/J_1^{[1]}={\mathcal O}$.

If $\lambda$ is a local section of ${\mathcal O}$ and $a$ a local section
of $F^r\mathcal O^{cris}_1$, $\lambda._{naive}a=\lambda^p._{nice}a$. In what follows,
we will always consider $J_1^{[r]}/J_1^{[r+1]}$ as a ${\mathcal
O}$-module via the nice structure (but we have to keep in mind that,
in the structure of $F$-zip, it is the naive structure which matters).

\bigskip
If $A$ is a ring,  $A<\theta_1,\theta_2,\ldots,\theta_d>$ is the
divided power algebra in the indeterminates
$\theta_1,\theta_2,\ldots,\theta_d$ with coefficients in A. If $t$
belongs to some divided power ideal in some $\mathbb Z_{(p)}$-algebra, and if
$m=qp+r$, with
$q,r\in\mathbb N$ and $r<p$, we have
$\gamma_{m}(t)=c_{m}t^r\gamma_q(\gamma_p(t))$ with $c_{m}=
{q!(p!)^q\over m!}$ a
unit in $\mathbb Z_{(p)}$. The following result is the
key for many explicit computations with $\cO^{cris}_1$:

\begin{proposition}\label{prop.good.algebra}
Let $k$ be a perfect ring of characteristic $p>0$, let ${\mathcal A}$
be a smooth $k$-algebra, $t_1,t_2,\ldots,t_d\in{\mathcal A}$ a
regular sequence, $A={\mathcal A}/(t_1^{p},t_2^p,\ldots,t_d^p)$ and $\ov
t_i$ the image of $t_i$ in A.
The unique homomorphism of divided power $A$-algebras
$$A<\theta_1,\theta_2,\ldots,\theta_d>\longrightarrow \mathcal O^{cris}_1(A)$$
sending $\theta_i$ to $\gamma_p(f(\ov t_i))$ is an isomorphism.
Moreover

i) For any $m=qp+r$ with $q,r\in\mathbb N$, $r<p$, set $\gamma_m(t_i)=
c_m{\ov t}_i^r\gamma_q(\theta_i)$. For any $r\in\mathbb N$,
$J_1^{[r]}(A)$ is the image of the sub-$A$-module generated by the
$\gamma_{m_1}(t_1)\gamma_{m_2}(t_2)\ldots\gamma_{m_d}(t_d)$ for $\Sigma
m_i\geq r$.

ii) Let $\ov I$ be the ideal of $A$ generated by the $\ov t_i$ and let
$\ov A=A/\ov I$. The $A$-module quotient $J_1^{[r]}(A)/J_1^{[r+1]}(A)$
is annihilated by $\ov I$ and is a free $\ov A$-module with the images
of the elements
$$
\gamma_{m_1}(t_1)\gamma_{m_2}(t_2)\ldots\gamma_{m_d}(t_d) \quad\mbox{ for }\quad \Sigma m_i= r
$$
as a basis.

iii) The natural map $J_1^{[r]}(A)/J_1^{[r+1]}(A)\to
(J_1^{[r]}/J_1^{[r+1]})(A)$ is injective and, for the nice
structure, $(J_1^{[r]}/J_1^{[r+1]})(A)$ is a free $A$-module with the images
of the elements
$$
\gamma_{m_1}(t_1)\gamma_{m_2}(t_2)\ldots\gamma_{m_d}(t_d) \quad\mbox{ for }\quad \Sigma m_i= r
$$
as a basis.

iv) Let $\rho$ be an endomorphism of the $k$-algebra $A$
such that one can find $\lambda_1,\lambda_2,\ldots,\lambda_d\in A$
with $\rho(\ov t_i)= \lambda_it_i)$. Then the endomorphism induced by
$\rho$ on $A<\theta_1,\theta_2,\ldots,\theta_d>$ (by functoriality
and the isomorphism $A<\theta_1,\theta_2,\ldots,\theta_d>\simeq
\mathcal O^{cris}_1(A)$)
is compatible with the divided power structure and sends $\theta_i$ to
$\lambda_i^p\theta_i$.
\end{proposition}

We start with a lemma:

\begin{lemma}\label{lem.good.algebra}
Let ${\mathcal A}'$ be a $k$-algebra, $t_1,t_2,\ldots,t_d\in {\mathcal
A}'$ a regular squence and  $A'= {\mathcal A}'/(t_1^p,t_2^p,\ldots,t_d^p)$.
Let ${\mathcal D}'$ the divided power envelope of ${\mathcal A}'$ with respect
to the ideal generate by the $t_i$'s and $\xi:{\mathcal A}'\to {\mathcal D}'$
the structural map. Then:

i) each $t_i^p$ belongs to the kernel of $\xi$ (and therefore ${\mathcal D}'$
may be viewed as an $A'$-algebra),

ii) the unique homomorphism of divided power $A'$-algebras
$$\eta: A'<\theta_1,\theta_2,\ldots,\theta_d>\to {\mathcal D}'$$
sending $\theta_i$ to $\gamma_p(t_i)$ is an isomorphism.
\end{lemma}

\ni {\it Proof of the lemma}: We have $t_i^p=p!\gamma_p(t_i)=0$, which
proves (i).

As an $A'$-algebra, ${\mathcal D}'$ is generated by the $\Pi
\gamma_{m_i}(t_i)$. But, if $m_i=pq_i+r_i$ with $q_i,r_i\in\mathbb N$ and
$r_i<p$, we have
$$\Pi \gamma_{m_i}(t_i)= \Pi c_{m_i}\Pi t_i^{r_i}\Pi
\gamma_{m_i}(\gamma_p(t_i))$$
hence ${\mathcal D}'$ is also generated by the $\Pi
\gamma_{m_i}(\gamma_p(t_i))$ and $\eta$ is surjective.

For $m\in\mathbb N$, let  ${\mathcal J}^{[m]}$ the $m$-th divided power of the
structural divided power ideal of ${\mathcal D}'$, let ${\mathcal I}$ the
ideal of ${\mathcal
A}'$ generated by the $t_i$'s and $\ov A'={\mathcal
A}'/{\mathcal I}$. As the $t_i$'s form a regular sequence, ${\mathcal I}/{\mathcal
I}^2$ is a free $\ov A'$-module with the image of the $t_i$'s as a
basis and the canonical map $\Gamma^{m}_{\ov A'}({\mathcal I}/{\mathcal
I}^2)\to {\mathcal J}^{[m]}/{\mathcal J}^{[m+1]}$ is an isomorphism, where $\Gamma^r_{\ov A'}(M)$
denotes the $r$-th divided power of an $\ov A'$-module $M$ \cite{Be} I, 3.4.4.
If ${\mathcal J}_1^{[m]}$ denote the inverse image under $\eta$ of ${\mathcal
J}^{[m]}$, we see that, for all $m$, the induced map
${\mathcal J}_1^{[m]}/{\mathcal J}_{1}^{[m+1]}\to {\mathcal J}^{[m]}/{\mathcal J}^{[m+1]}$
is an isomorphism. As $\cap_{m\in\mathbb N}{\mathcal J}_1^{[m]}=0$, the map
$\eta$ is injective. $\square$

\medskip
\ni {\it Proof of proposition \ref{prop.good.algebra}}
Let ${\mathcal A}'={\mathcal A}$, but viewed as an ${\mathcal A}$-algebra via the
absolute Frobenius, that we use to identify ${\mathcal A}$ to a sub-ring
of ${\mathcal A}'$. Therefore, any element of ${\mathcal A}$ has a unique
$p$-th root in ${\mathcal A'}$ and, for an $b\in {\mathcal A}'$, we have
$b^p\in {\mathcal A}$.
We denote by $t'_i$ the element $t_i$ viewed as an
element of ${\mathcal A}'$, hence $(t'_i)^p=t_i$.

As ${\mathcal A}$ is smooth, ${\mathcal A}'$ is a syntomic ${\mathcal A}$-algebra
and $t_1,t_2,\ldots,t_d$ is still a regular sequence in ${\mathcal A}'$.
The map $\varphi:{\mathcal A}'\to {\mathcal A}$ sending $b$ to $b^p$ is
an isomorphism and the compositum $\ov\varphi$ with the projection
onto $A$ is surjective, with kernel the ideal
generated by the $t_i$'s. As ${\mathcal A}'$ is smooth, if ${\mathcal D}'$ is
as in the lemma, we have (ref. XXX) an exact sequence
$$0\to \mathcal O^{cris}_1(A)\to {\mathcal D}'\to {\mathcal D'}\otimes_{{\mathcal A}'}\Omega_{{\mathcal
A}'/k}^1$$
which, granted to the previous lemma, can be rewritten as

$$0\to \mathcal O^{cris}_1(A)\to A'<\theta_1,\theta_2,\dots,\theta_d>\to
A'<\theta_1,\theta_2,\dots,\theta_d>\otimes_{{\mathcal A}'}\Omega_{{\mathcal
A}'/k}^1$$
with $\theta_i$ mapping to $\gamma_p(t_i^p)$. Therefore, we have
$d\theta_i= \gamma_{p-1}(t_i^p).pt_i^{p-1}=0$. For
$m=(m_1,m_2,\ldots,m_d)\in\mathbb N^d$ set
$\gamma_m(\theta)=\Pi\gamma_{m_i}(\theta_i)$.
We have also $d\gamma_m(\theta)=0$, for all $m$.

But $A'<\theta_1,\theta_2,\dots,\theta_d>$ is a free $A'$-module with
basis the $\gamma_m(\theta)$'s.
If $\sum a_{m}\gamma_{m}(\theta)\in A'<\theta_1,\theta_2,\dots,\theta_d>$,
we have $d(\sum a_{m}\gamma_{m}(\theta))=\sum da_{m}\gamma_m(\theta)$.
It is easy to check that we have
$\Omega_{A'/k}^1=A'\otimes_{{\mathcal A}'}\Omega_{{\mathcal A}'/k}^1$ and
that the sequence
$$0\to A\to A'\to \Omega_{A'/k}^1$$
is exact. This proves that the map
$$A<\theta_1,\theta_2,\dots,\theta_d>\to\mathcal O^{cris}_1(A)$$
is an isomorphism.

The proof of (i) and (ii) and (iv) are straightforward. Let us prove
(iii). We have to understand the sheaf $F$ associated to
the presheaf $P: X\mapsto J_{1}^{[r]}(X)/J_{1}^{[r+1]}(X)$. It is
enough to consider the restriction of the functor $P$ to the objects
$X$ of ${\mathcal C}$ of the form $X=Spec\ A$, with $A$ as in the
proposition.

Let $M$ be the set of the $\u m=\{m_1,m_2,\ldots,m_d\}\in\mathbb N^{d}$ such
that $\sum
m_i=r$.  For any such $\u m$,  let $\gamma_{\u m}$ the image of $\Pi
\gamma_{m_i}(t_i)$ in $P(A)$ and $\tilde\gamma_{\u m}$ its image in
$F(A)$. Checking carefully, we see that the map $P(A)\to
F(A)$ is $A$-linear if $F(A)$ is equipped with the naive structure.
We may express this fact as follows: Let $A_1$ the image of $\varphi:
A\to A$. The absolute Frobenius induces an isomorphism $f: \ov A\to
A_1$ and we may view it to view $P(A)$ as a free $A_1$-module. Now
the map $P(A)\to FA)$ is $A_1$-linear when we endow $F(A)$ with the
$A_1$-structure coming from the inclusion $A_1\subset A$ and the nice
structure of $A$-module on $F(A)$. Therefore, if we set
$P'(A)=A\otimes_{A_1}P(A)$, we have a natural $A$-linear map
$P'(A)\to F(A)$ and what we want to prove is that this map is an
isomorphism.

Because there are enough algebras of this kind, it makes sense to
speak of the sheaf associated to the presheaf $A\to P'(A)$ and all
what we have to prove is that this presheaf is a sheaf. We are easily
reduced to checking that, if $B$ is any faithfully flat $A$-algebra of
the type $A[x_1,x_2,\ldots,x_r]/(u_1,u_2,\ldots,u_n)$ with
$u_1,u_2,\ldots,u_n\in A[x_1,x_2,\ldots,x_r]$ a sequence which is
transverse regular with respect to $A$, then the sequence
$$0\to P'(A)\to P'(B)\begin{matrix}\to\\\to\end{matrix} P'(B\otimes_{A}B)$$
is exact. Replacing $B$ with a covering if necessary, we may assume
that $u_i=v_i^p$ and the $v_i^p$ 's are transverse regular as well.
Set  ${\mathcal B}={\mathcal A}[x_1,x_2,\ldots,x_r]$ and ${\mathcal C}={\mathcal
B}\otimes_{{\mathcal A}}{\mathcal B}= {\mathcal
A}[x_1,x_2,\ldots,x_r,y_1,y_2,\ldots,y_r]$ if we
set $x_i=x_i\otimes 1$ and $y_i=1\otimes x_i$. If for $1\leq i\leq
n$, we choose a lifting $t_{d+i}$ of $v_i$ in ${\mathcal B}$ and if, for
$1\leq i\leq n$,  we set $t_{d+i}=t_{d+i}\otimes 1$ and
$t_{d+n+i}=1\otimes t_{d+i}$, The sequence $(t_i)_{1\leq i\leq d}$
(resp. $(t_i)_{1\leq i\leq d+n}$, $(t_i)_{1\leq i\leq d+2n}$)
is regular in ${\mathcal A}$ (resp. ${\mathcal B}$, resp. ${\mathcal C}$) and
$$
A= {\mathcal A}/((t_i^p)_{1\leq i\leq d})\ , \ B= {\mathcal B}/((t_i^p)_{1\leq
i\leq d+n})\ , \ B\otimes_AB = {\mathcal C}/((t_i^p)_{1\leq i\leq d+2n})
$$
Let $M$ (resp. $N$, resp. $L$) the set of the $\u m=(m_i)_{1\leq i\leq d}$
(resp. $(m_i)_{1\leq i\leq d+n}$, resp. $(m_i)_{1\leq i\leq d+2n}$)
such that $\sum m_i= r$ and, for $\u m \in M$ (resp. $N$, $L$), let
$\gamma_{\u m}(t)$ the image of $\Pi\gamma_{m_i}(t_i)$ in $P'(A)$
(resp. $P'(B)$, $P'(B\otimes_AB)$). These $\gamma_{\u m}(t)$ form a
basis respectively of the free $A$-module $P'(A)$, the free
$B$-module $P'(B)$, the free $B\otimes_AB$-module $P'(B\otimes_AB)$
and we are reduced to prove the exactness of the sequence
$$0\to \oplus_{\u m\in M}A\gamma_{\u m}(t)\to
\oplus_{\u m\in N}B\gamma_{\u m}(t)\begin{matrix}\to\\\to\end{matrix}
\oplus_{\u m\in L}B\otimes_AB\gamma_{\u m}(t)$$
This is easily reduced to the exactness of
$$0\to A\to B\to B\otimes_A B$$
which comes from the fact that $B$ is faithfully flat over $A$.
$\square$

\medskip
Call $J$ the kernel of the absolute Frobenius $\varphi$ on ${\mathcal O}$ (hence,
with the usual notations $J=\alpha_p$).

\begin{proposition}\label{prop.comm.diag}
We have a commutative diagram of abelian sheaves
$$\begin{matrix}0 & \to & J & \to & {\mathcal O}&\mathop{\to}\limits^{\varphi}&{\mathcal O}&\to&0\\
&  & \downarrow && \llap{}\left\downarrow
\vbox to 3mm{}\right.\rlap{$\scriptstyle f$} & &\parallel\\
0 & \to & J_1^{[1]} & \to & \mathcal O^{cris}_1&\to&{\mathcal O}&\to& 0\end{matrix}$$
whose lines are exact and whose vertical arrows are monomorphisms.
\end{proposition}

\begin{proof}
The commutativity of the diagram is clear.
We already know that the second line is exact. As our topology allows
us to extract $p$-th roots, the Frobenius is an epimorphism and the
first line is also exact. We are left to check that $f$ is a monomorhism.
As we can always cover any object $X$ of ${\mathcal C}$ by affine
$k$-schemes $Y=Spec\ A$, with $A$ as in the previous proposition, we
are reduced to check the injectivity of $A\to\mathcal O^{cris}_1(A)$ which is a
consequence of that proposition. $\square$
\end{proof}

\medskip
If $A$ is a commutative ring and $M$ is an $A$-module, we may
consider the symmetric algebra $Sym_AM=\oplus_{r\in\mathbb N}Sym^r_AM$ and
the divided power algebra $\Gamma_AM=\oplus_{r\in\mathbb N}\Gamma_A^rM$.
They are graded algebras, with grading indexed by $\mathbb N$. For
$r\in\mathbb N$, $\Gamma_A^rM$ is the sub-$A$-module of $\Gamma_AM$ generated
by the $\gamma_{r_1}(x_1)\gamma_{r_2}(x_2)\ldots\gamma_{r_d}(x_d)$ with
$r_1,r_2,\ldots,r_d\in\mathbb N$ satisfying $\Sigma r_i=r$ and
$x_1,x_2,\ldots,x_d\in M$. This extends to rings and modules in a topos.

\begin{proposition}\label{prop.intersection.Fr}
We have $\cap_{r\in\mathbb N}J_1^{[r]}=0$ Moreover, for any $r\in\mathbb N$, the
obvious map
$$\Gamma_{{\mathcal O}}^rI\to J_1^{[r]}/J_1^{[r+1]}$$
is an isomorphism of ${\mathcal O}$-modules.
\end{proposition}

\begin{proof} From proposition \ref{prop.good.algebra}, we see that for any $k$-algebra $A$
satisfying the condition of \ref{prop.good.algebra}, we have
$\cap_{r\in\mathbb N}J_1^{[r]}(A)=0$ and the map $\Gamma^r_{{\mathcal O}}I(A)\to
(J_1^{[r]}/J_1^{[r+1]})(A)$ is an isomorphism. The proposition
thereof follows from the fact that any object $X$ of ${\mathcal C}$ has a
covering  by affine
$k$-schemes $Y=Spec\ A$, with $A$ of this kind.  $\square$
\end{proof}

\medskip
In view of proposition \ref{prop.comm.diag}, we may use $f$ to regard ${\mathcal O}$ as a
sub-ring of $\mathcal O^{cris}_1$ and
$J$ as a sub ${\mathcal O}$-Module of $\mathcal O^{cris}_1$. For any
$r\in\mathbb N$, we call $\tilde{F}_r\mathcal O^{cris}_1$ the sub-${\mathcal O}$-module of
$\mathcal O^{cris}_1$ generated locally by the elements of the form
$\gamma_{pr_1}(x_1)\gamma_{pr_2}(x_2)\ldots\gamma_{pr_d}(x_d)$ with
$r_1,r_2,\ldots,r_d\in\mathbb N$ satisfying $\Sigma r_i\leq r$ and
$x_1,x_2,\ldots,x_d\in J$. We have $\tilde{F}_0\mathcal O^{cris}_1= {\mathcal O}$
and we set $\tilde{F}_{-1}\mathcal O^{cris}_1=0$.

The next theorem gives the second construction
for the structure of $F$-zip on $\mathcal O^{cris}_1$~:

\begin{theorem}\label{thm.cartier.iso.1}
i) We have $\cup_{r\in\mathbb N}\tilde{F}_r\mathcal O^{cris}_1=\mathcal O^{cris}_1$.

ii) (Cartier isomorphism, second version) Let $r\in\mathbb N$. There is a unique morphism of ${\mathcal O}$-modules
$$f_r: J_1^{[r]}/J_1^{[r+1]}\to \tilde{F}_r\mathcal O^{cris}_1/\tilde{F}_{r-1}\mathcal O^{cris}_1$$
such that, if $r_1,r_2,\ldots,r_d\in\mathbb N$ satisfies $\Sigma r_i=r$ and if
$x_1,x_2,\ldots,x_d\in J_1$, then $f_r$ sends the image of
$\gamma_{r_1}(x_1)\gamma_{r_2}(x_2)\ldots\gamma_{r_d}(x_d)$ to
the image of
$(-1)^r\gamma_{pr_1}(x_1)\gamma_{pr_2}(x_2)\ldots\gamma_{pr_d}(x_d)$.

Moreover $f_r$ is an isomorphism.
\end{theorem}

\begin{proof}
Again, it is enough to prove the corresponding assertion for
the sections of these sheaves with values in $A$ for any $k$-algebra
$A$ of the kind considered in proposition \ref{prop.good.algebra}.

We use the notation of that proposition. Then $J(A)$ is the $A$-module
generated by the $\ov t_i$. For $m_1,m_2,\ldots,m_d\in\mathbb N$, we have
$$
\gamma_{pm_1}(\ov t_1)\gamma_{pm_2}(\ov t_2)\ldots\gamma_{pm_d}(\ov t_d)=
c_{m_1,m_2,\ldots,m_d}\gamma_{m_1}(\theta_1)\gamma_{m_2}(\theta_2)\ldots\gamma_{m_d}(\theta_d)
$$
with $c_{m_1,m_2,\ldots,m_d}=\Pi {(pm_i)!\over m_i!(p!)^{m_i}}$, a
unit of $\mathbb Z_{(p)}$. Hence,  $\tilde{F}_r\mathcal O^{cris}_1(A)$ is the free
$A$-module with basis the
$\gamma_{m_1}(\theta_1)\gamma_{m_2}(\theta_2)\ldots\gamma_{m_d}(\theta_d)$
with $\Sigma m_i\leq r$ and the first
assertion is clear.
We see also that the $A$-module $\tilde{F}_r\mathcal O^{cris}_1(A)/\tilde{F}_{r_1}\mathcal O^{cris}_1(A)$ is free
with the images of the
$\gamma_{m_1}(\theta_1)\gamma_{m_2}(\theta_2)\ldots\gamma_{m_d}(\theta_d)$
  (or, equivalently of the $\gamma_{pm_1}(\ov t_1)\gamma_{pm_2}(\ov
t_2)\ldots\gamma_{pm_d}(\ov t_d)$) for $\Sigma m_1=r$ as a basis.

Let $ x,y\in J(A)$. For all $m\in\mathbb N$, we have
$\gamma_{pm}(x+y)=\Sigma_{i+j=m}\gamma_{pi}(x)\gamma_{pj}(y)+e(x,y)$
with $e(x,y)\in \tilde{F}_{m-1}\mathcal O^{cris}_1(A)$. From that and from the fact that if
$a\in \tilde{F}_m\mathcal O^{cris}_1(A)$
and $b\in \tilde{F}_n\mathcal O^{cris}_1(A)$, then $xy\in
\tilde{F}_{n+m}\mathcal O^{cris}_1(A)$, we deduce the fact that $f_r$ is well defined. The
$A$-module $(J_1^{[r]}/J_1^{[r+1]})(A)$ is free with the images of the
$\gamma_{m_1}(\ov t_1)\gamma_{m_2}(\ov
t_2)\ldots\gamma_{m_d}(\ov t_d)$ for $\sum m_i=r$ as a basis.  As the
$A$-linear map $f_r$ sends the basis onto a basis of the free
$A$-module $\tilde{F}_r\mathcal O^{cris}_1(A)/\tilde{F}_{r_1}\mathcal O^{cris}_1(A)$, $f_r$ is bijective.
$\square$
\end{proof}

\begin{remarks}\label{rem.cartier} - i) The reason for the sign $(-1)^r$ is that we want
$f_r$ to be ``morally'' the Frobenius divided out by $p^r$. In
characteristic $0$, we have $\gamma_m(x^p)= p^mu_m\gamma_{pm}(x)$,
with $u_m\in\mathbb Z_{(p)}$, congruent to $(-1)^m\hbox{ mod }p$.

ii) We may say that $\mathcal O^{cris}_1$ is a {\it ring object in the category of
$F$-zips}, i.e this is a ring, which is a ${\mathcal O}$-Algebra and for
$r,s\in\mathbb N$, we have $\tilde{F}^r\mathcal O^{cris}_1\times \tilde{F}^{s}\mathcal O^{cris}_1\subset \tilde{F}^{r+s}\mathcal O^{cris}_1$ and
$\tilde{F}_{r}\mathcal O^{cris}_1\times \tilde{F}_s\mathcal O^{cris}_1\subset \tilde{F}_{r+s}\mathcal O^{cris}_1$.
\end{remarks}

\subsection{\it End of the proof of theorem \ref{thm.phi.onD.iso} and of the construction of the structure of
an $F$-zip on $\cO^{cris}_1$}\label{ssec-6.5}

\medskip
By the flatness of the $p$-adic ring $\mathcal G$ it suffices to show theorem \ref{thm.phi.onD.iso}
for $n=1$.

First it follows inductively from theorem \ref{thm.cartier.iso.1} (ii) and remark \ref{rem.cartier} (i)
that we have, for all $r\in\n$,
$$
\varphi(J^{[r]}_1) = F_r\mathcal O^{cris}_1\,,
$$
i.e., $\tilde{F_r} = F_r$, and that we have $\tilde{F}^r = F^r$ (for the rings considered in the previous section):
The start of the induction is the fact that we have $F^0 = \cO^{cris}_1 = \tilde{F}^0$ by definition,
and that the Frobenius $\varphi= \varphi_0: \cO^{cris}_1 \longrightarrow \cO^{cris}_1$ has image
$F_0\cO^{cris}_1 = \cO \mathop{\hookrightarrow}\limits^{f} \cO^{cris}_1$ and kernel $J^{[1]}_1$.
We conclude that $F_0 = \cO = \tilde{F}_0$ and $F^1 = ker(\varphi) = \tilde{F}^1$.
Then, by remark \ref{rem.cartier} (i) the map $f_r$ in theorem \ref{thm.cartier.iso.1} (ii) can
be identified with the map induced by $\varphi_r$, which gives the induction steps.
In fact, if $F_s = \tilde{F}_s$ for $s\leq r-1$ and $F^s = \tilde{F}^s$ for $s\leq r$, then the
two exact sequences
$$
0 \longrightarrow F^{r+1} \longrightarrow F^r \mathop{\longrightarrow}\limits^{\varphi_r} F_r/F_{r-1} \longrightarrow 0
$$
$$
0 \longrightarrow \tilde{F}^{r+1} \longrightarrow \tilde{F}^r \mathop{\longrightarrow}\limits^{\varphi_r} \tilde{F}_r/\tilde{F}_{r-1} \longrightarrow 0 $$
imply that $F_s = \tilde{F}_s$ for $s\leq r$ and $F^s = \tilde{F}^s$ for $s\leq r+1$.

This now shows that $\cap_r F^r = 0$ (by proposition \ref{prop.intersection.Fr}) and that
$\cup_r F_r = \cO^{cris}_1$ (by theorem \ref{thm.cartier.iso.1} (i)), showing the surjectivity
in theorem \ref{thm.phi.onD.iso}, and the remaining property of $F$-zips.

\section{$\varphi$-gauge-cohomology - a refinement of crystalline cohomology}\label{sec-7}

Let $k$ be a perfect field of characteristic $p$, let $W_n=W_n(k)$ for $n\in \n$,
and let $X$ be a syntomic variety over $k$.

\subsection{\it The definition of gauge-cohomology}\label{ssec-7.1}

\medskip
We define the $i$-th $\varphi$-gauge cohomology of level $n$ of $X$ by
$$
H^i_g(X,W_n) = H^i_{syn}(X,\cG_n)\,
$$
This is a $\varphi$-$W_n$-gauge by transport of structure: We let $H^i_g(X,W_n)^r :=
H^i_{syn}(X,\cG^r_n)$, and let the structural maps
$$
H^i_g(X,W_n)^r \mathop{\toto}\limits^{f}_{v} H^i_g(X,W_n)^{r+1}
$$
be induced by the morphisms
$$
\cG^r_n \mathop{\toto}\limits^{f}_{v} \cG^{r+1}_n\,.
$$
Moreover, the $\sigma$-linear map of $W_n$-modules
$$
\varphi: H^i_g(X,W_n)^{+\infty} = \lim\limits_{\stackrel{\textstyle \longrightarrow}{r\mapsto +\infty},f} H^i_g(X,W_n)^r \longrightarrow \lim\limits_{\stackrel{\textstyle \longrightarrow}{r\mapsto -\infty},v} H^i_g(X,W_n) = H^i_g(X,W_n)^{-\infty}
$$
is induced by the isomorphism $\varphi: \cG^{+\infty}_n \longrightarrow \cG^{-\infty}_n$, i.e., the isomorphisms
$$
\lim\limits_{\stackrel{\textstyle \longrightarrow}{r\mapsto +\infty},f} H^i_{syn}(X,\cG^r_n)
\cong H^i_{syn}(X,\cG^{+\infty}_n) \mathop{\longrightarrow}\limits^{\varphi}_{\sim} H^i_{syn}(X,\cG^{-\infty}_n)
\cong \lim\limits_{\stackrel{\textstyle \longrightarrow}{r\mapsto -\infty},v} H^i_{syn}(X,\cG^r_n)
$$
where the outer isomorphisms come from the commuting of cohomology with limits.

We note that
$$
H^i_g(X,W_n)^{-\infty} = H^i_{syn}(X,\cG^{-\infty}) = H^i_{syn}(X,\cO^{cris}) \cong H^i_{cris}(X/W_n)
$$
by the comparison theorem of Fontaine and Messing \cite{FM}. In this way, we achieved the refinement of
crystalline cohomology announced in section \ref{ssec-2.3}. We note also that
$$
H^i_g(X,W_n)^{r-1} \mathop{\longleftarrow}\limits^{v}_{\sim} H^i_g(X,W_n)^r
$$
is an isomorphism for $r\leq 0$, because this holds for $\cG^{r-1}_n\mathop{\leftarrow}\limits^{v} \cG^r_n$. Hence the gauge $H^i_g(X,W_n)$
is effective (concentrated in degrees $\geq 0$), and we have $H^i_g(X,W_n)^{-\infty} = H^i_g(X,W_n)^0
= H^i_{cris}(X/W_n)$.

\subsection{\it The de Rham gauge of a variety $X$ over $k$}\label{ssec-7.2}

\medskip
We denote by ${\mathcal C}^{(p)}_b({\mathcal O}_X)$ the following category:

-- An object is a bounded complex
$$
\ldots\to C^{n-1}\to C^n \,\smash{\mathop{\hbox to 4mm{\rightarrowfill}}
\limits^{\scriptstyle d^n}}\, C^{n+1}\to\ldots\leqno
(C\raisebox{5pt}{$\cdot$})
$$
of ${\mathcal O}_X$-modules of finite type, whose
differentials (which are additive) are $\cO_X^p$-linear (i.e., satisfy $d^n(a^px)=a^pd^nx$ for
any $n\in\Bbb Z$, any local section $a$ of ${\mathcal O}_X$ and $x$ of $C^n$ for any $n\in\Bbb Z$).

-- A morphism $\alpha :C\raisebox{5pt}{$\cdot$} \to
D\raisebox{5pt}{$\cdot$}$ is a collection of ${\mathcal
O}_X$-linear maps $\alpha^{n} :C^{n}\to D^{ n}$, for $n\in\Bbb Z$,
such that the diagram
$$
\begin{array}{ccccccccc}
\ldots & \to & C^{n-1} & \to & C^n & \to & C^{n+1} & \to &
\ldots\\
&  & \da & & \da & & \da & \\
\ldots & \to & D^{n-1} & \to & D^n & \to & D^{n+1} & \to & \ldots
\end{array}
$$
is commutative.

This is an abelian ($\cO_X^p$-linear) category.

\medbreak

Because $k$ is perfect, we have
$\Omega^1_{X/k}=\Omega^1_X:=\Omega^1_{X/\Bbb Z}$. The de Rham
complex
$$
{\mathcal
O}_X\to\Omega_X^1\to\ldots\to\Omega_X^{n-1}\to\Omega_{X}^n\to\Omega_X^{n+1}\to\ldots
\leqno (\Omega\raisebox{5pt}{$\cdot$}_X)
$$
is an object of the above category (We adopt the following convention: if a
complex $C\raisebox{5pt}{$\cdot$}$ starts with a given term on the
left, this term is in degree $0$ and $C^n=0$ for $n<0$).

\medbreak
The restriction of scalars via the absolute Frobenius $\sigma:{\mathcal O}_X\to{\mathcal O}_X$
defines an additive, exact and faithful functor from  ${\mathcal C}^{(p)}_b({\mathcal O}_X)$ to itself:
If $C\raisebox{5pt}{$\cdot$}$ is an object of ${\mathcal C}^{(p)}_b({\mathcal O}_X)$, we let
$(_{\sigma}\!C)^n=_{\sigma}\!(C^n)$ and take the same differentials.
In other words, the underlying complex of sheaves of abelian
groups is $C\raisebox{5pt}{$\cdot$}$, but we change the structure
of ${\mathcal O}_X$-module on each $C^n$ : the multiplication of a
local section $x$ of $_{\sigma}\! C^n$ by a local section $a$ of
${\mathcal O}_X$ is the multiplication of $x$ by $a^p$ in $C^n$.

\medbreak
Let us recall that, for any integer $n\in\Bbb N$, the {\it Cartier
isomorphism} is an isomorphism of sheaves of abelian groups
$$
c: {\mathcal H}^n(\Omega_X\raisebox{5pt}{$\cdot$})\to\Omega_X^n
$$
satisfying $c(a^p\alpha)=a c(\alpha)$ for $a$
(resp. $\alpha$) any local section of ${\mathcal O}_X$ (resp.
${\mathcal H}^n(\Omega_X\raisebox{5pt}{$\cdot$})$). It is
characterized by the fact that, if $a_0,a_1,\ldots,a_n$ are local
sections of ${\mathcal O}_X$, then
$c^{-1}(a_0da_1\wedge\ldots\wedge da_n)$ is the image of the closed
form $a_0^pa_1^{p-1}a_2^{p-1}\ldots a_n^{p-1}da_1\wedge
da_2\wedge\ldots\wedge da_n$.

We can also regard $c$ as an $\cO_X$-linear morphism
$$
_{\sigma}\!{\mathcal H}^n(\Omega_X\raisebox{5pt}{$\cdot$})\to\Omega_X^n\,.
$$
If $Z\Omega_X^n$ denotes the kernel of $d : \Omega_X^n\to\Omega_X^{n+1}$, the map $c$ induces an
exact sequence of sheaves of $\cO_X$-modules
$$
_{\sigma}\!\Omega_X^{n-1} \mathop{\longrightarrow}\limits^{\scriptstyle d}\,_{\sigma}\!Z\Omega_X^n
\mathop{\longrightarrow}\limits^{\scriptstyle c}\,\Omega_X^n\to 0\,.
$$
We observe that the fact that $d(a^p)\omega=a^pd\omega$ (for
$a$ local section of ${\mathcal O}_X$ and $\omega$ local section
of $\Omega_X^n$) implies that $_{\sigma}\!Z\Omega_X^n$ is in fact a sub
${\mathcal O}_X$-module of $_{\sigma}\!\Omega_X^n$.

\bigskip
We are now ready to construct {\it the de Rham $\varphi$-gauge $G_1(X)= (G_1\raisebox{5pt}{$\cdot$}(X),f,v,\varphi)$ of $X$},
which is an object in ${\mathcal C}^{(p)}_b({\mathcal O}_X)$:

\bigskip\noindent
-- For $r<0$, $G^r_1(X)$ is the de Rham complex $\Omega\raisebox{5pt}{$\cdot$}_X$:
$$
{\mathcal
O}_X\stackrel{d}{\to}\Omega_X^1\stackrel{d}{\to}\ldots\to
\Omega_X^{n-1}\stackrel{d}{\to}\Omega_{X}^n\stackrel{d}{\to}\Omega_X^{n+1}
\stackrel{d}{\to}\ldots
$$
-- for $r\geq 0$, $G^r_1(X) $ is the complex
$$
_{\sigma}\!{\mathcal
O}_X\stackrel{d}{\to}_{\sigma}\!\Omega_X^1\stackrel{d}{\to}
\ldots\stackrel{d}{\to}_{\sigma}\!\Omega_X^{r-1}\stackrel{d}{\to}
_{\sigma}\!Z\Omega_{X}^r\,\smash{\mathop{\hbox to
5mm{\rightarrowfill}} \limits^{\scriptstyle
dc}}\,\Omega_X^{r+1}\stackrel{d}{\to}\ldots\stackrel{d}{\to}
\Omega_X^n\stackrel{d}{\to}\ldots
$$
-- the map $f: G^r_1(X) \to G^{r+1}_1(X)$ is $0$ for $r< 0$ and is, for $r\geq 0$,
$$
\begin{array}{ccccccccccccccc}
_{\sigma}\!{\mathcal O}_X & \stackrel{d}{\to} & \ldots &
_{\sigma}\!\Omega_X^{r-1} & \stackrel{d}{\to} & &
_{\sigma}Z\Omega_{X}^r & \,\smash{\mathop{\hbox to
5mm{\rightarrowfill}} \limits^{\scriptstyle dc}}\, &
\Omega_X^{r+1} & \stackrel{d}{\to} & \Omega_X^{r+2} &
\stackrel{d}{\to} & \Omega_X^{r+3} &
\ldots\\
 \llap{}\left\da
\vbox to 4mm{}\right.\rlap{$\scriptstyle
\parallel$} &&& \llap{}\left\da
\vbox to 4mm{}\right.\rlap{$\scriptstyle \parallel$} &&&
\llap{}\left\da \vbox to 4mm{}\right.\rlap{$\scriptstyle incl$} &&
\llap{}\left\da \vbox to 4mm{}\right.\rlap{$\scriptstyle 0$}
 && \llap{}\left\da
\vbox to 4mm{}\right.\rlap{$\scriptstyle 0$} && \llap{}\left\da
\vbox to 4mm{}\right.\rlap{$\scriptstyle 0$}\\
_{\sigma}\!{\mathcal
O}_X&\stackrel{d}{\to}&\ldots&_{\sigma}\!\Omega_X^{r-1}&\stackrel{d}{\to}&
&_{\sigma}\!\Omega_{X}^r&\stackrel{d}{\to}&_{\sigma}Z\Omega_X^{r+1}&
\,\smash{\mathop{\hbox to 5mm{\rightarrowfill}}
\limits^{\scriptstyle
dc}}\,&\Omega_X^{r+2}&\stackrel{d}{\to}&\Omega_X^{r+3}&\ldots\\
\end{array}
$$
-- the map $v: G^{r+1}_1(X)\to G^r_1(X) $ is the identity for $r<-1$ and is, for $r\geq -1$,
$$
\begin{array}{cccccccccccccc}
_{\sigma}\!{\mathcal
O}_X&\stackrel{d}{\to}&\ldots&_{\sigma}\!\Omega_X^{r-1}&\stackrel{d}{\to}&
&_{\sigma}Z\Omega_{X}^r&\,\smash{\mathop{\hbox to
5mm{\rightarrowfill}} \limits^{\scriptstyle
dc}}\,&\Omega_X^{r+1}&\stackrel{d}{\to}&\Omega_X^{r+2}&\stackrel{d}{\to}&\Omega_X^{r+3}&\ldots\\
\llap{}\left\uparrow \vbox to 4mm{}\right.\rlap{$\scriptstyle 0$}
&&& \llap{}\left\uparrow \vbox to 4mm{}\right.\rlap{$\scriptstyle
0$} &&& \llap{}\left\uparrow \vbox to
4mm{}\right.\rlap{$\scriptstyle 0$} && \llap{}\left\uparrow \vbox
to 4mm{}\right.\rlap{$\scriptstyle c$} && \llap{}\left\uparrow
\vbox to 4mm{}\right.\rlap{$\scriptstyle
\parallel$} && \llap{}\left\uparrow
\vbox to 4mm{}\right.\rlap{$\scriptstyle \parallel$}\\
_{\sigma}\!{\mathcal
O}_X&\stackrel{d}{\to}&\ldots&_{\sigma}\!\Omega_X^{r-1}&\stackrel{d}{\to}&
&_{\sigma}\!\Omega_{X}^r&\stackrel{d}{\to}&_{\sigma}Z\Omega_X^{r+1}&
\,\smash{\mathop{\hbox to 5mm{\rightarrowfill}}
\limits^{\scriptstyle
dc}}\,&\Omega_X^{r+2}&\stackrel{d}{\to}&\Omega_X^{r+3}&\ldots
\end{array}
$$
Clearly $vf=fv=0$, hence we have a gauge. We see also that for all
$r\in\Bbb Z$, the map $(f,v): G^r_1(X) \to G^{r+1}_1(X) \oplus G^{r-1}_1(X)$ is injective.
Hence we have a strict gauge.

Because $c:Z_{\sigma}\!{\mathcal O}_X\to{\mathcal O}_X$ is an
isomorphism, $v_{-1}$ is an isomorphism, hence we have an
effective gauge.

\noindent
--If $d$ is the dimension of $X$, we have
$$
G^r_1(X) =_{\sigma}\!\Omega_X\raisebox{5pt}{$\cdot$} \quad\mbox{ for }\quad r\geq d
$$
(for $r=d$, because $_{\sigma}Z\Omega_X^d=_{\sigma}\!\Omega_X^d$) and
$f: G^r_1(X) \to G^{r+1}_1(X)$ is an isomorphism if $r\geq d$.

Hence, we have a gauge concentrated in the interval $[0,d]$, with $G^{-\infty}_1(X) =\Omega_X\raisebox{5pt}{$\cdot$}$
and $G^{+\infty}_1(X)=_{\sigma}\!\Omega_X\cdot$. We define the isomorphism
$$
\varphi: G^{+\infty}_1(X)\to_{\sigma}\!G^{-\infty}_1(X)
$$
as the identity on $_{\sigma}\!\Omega_X\raisebox{5pt}{$\cdot$}$, which finishes the definition of the de Rham
$\varphi$-gauge $G_1(X)$.

\begin{theorem}\label{thm.deRham.gauge}
Let $\alpha: X_{syn} \longrightarrow X_{Zar}$ be the morphism of sites coming from the fact
that the Zariski topology is coarser than the syntomic cohomology. There is a canonical isomorphism
$$
R\alpha_\ast \cG_1 \mathop{\longrightarrow}\limits^{\sim} G_1\;.
$$
\end{theorem}

\begin{proof}
This follows via the same methods proving the Fontaine-Messing isomorphism
$$
H^i_{syn}(X,\cO^{cris}_1) \cong H^i_{cris}(X/k)=H^i_{dR}(X/k)\,,
$$
using canonical isomorphisms following from results of Berthelot \cite{Be}
$$
R\alpha_\ast I^{[r]}_1 \cong \Omega^{\geq r}_X\,,
$$
where the complex on the right is obtained by naive truncation, i.e., is the upper part,
starting with $\Omega_X^r$, of the de Rham complex.
\end{proof}

\begin{corollary}\label{cor.gauge.coh}
Let $X$ be a proper variety over $k$. Then the following holds.

\smallskip\noindent
(a) The gauge cohomology $H^i_g(X,W_n)$ is of finite type and is concentrated in the interval $[0,i]$,
and it vanishes for $i \geq d$.

\smallskip\noindent
(b) One has $H^i_g(X,W_n)^0 = H^i_{cris}(X/W_n)$.

\smallskip\noindent
(c) If $X$ is smooth, proper, and irreducible of pure dimension $d$, then the Poincar\'e duality for crystalline
cohomology extends to a perfect duality of $\varphi$-$W_n$-gauges
$$
H^i_g(X,W_n) \times H^{2d-i}_g(X,W_n) \longrightarrow H^{2d}_g(X,W_n) \mathop{\longrightarrow}\limits^{\sim} W_n(-d)\,.
$$
Here the $W_n$-gauge $W_n(-d)$ extends to a $\varphi$-$W_n$-gauge by defining
$$
\varphi: W_n(-d)^{+\infty} = W_n \rightarrow W_n = W_n(-d)^{-\infty}
$$
as the Frobenius $\sigma$ on $W_n$, which is $\sigma$-linear.
\end{corollary}

\begin{proof}
All questions are easily reduced to the case $n=1$.
But by theorem \ref{thm.deRham.gauge}, we have $H^i_g(X,W_1) = H^i_{Zar}(X,G^r_1)$, and thus the claim
follows from classical Serre duality, and the explicit shape of the de Rham gauge $G_1(X)$ defined above.
\end{proof}

\end{document}